\documentclass{article}
\usepackage[a4paper, total={6in, 8in}]{geometry}

\usepackage{parskip}
\usepackage{amsfonts}
\usepackage{amsmath}
\usepackage{amsthm}
\usepackage{mathtools}
\usepackage{graphicx}
\usepackage{amssymb}
\usepackage{hyperref}
\usepackage{comment}

\newtheorem{theorem}{Theorem}[section]
\newtheorem{lemma}[theorem]{Lemma}
\newtheorem{proposition}[theorem]{Proposition}
\newtheorem{definition}[theorem]{Definition}

\newcommand{\ZZ}{\mathbb{Z}}
\newcommand{\QQ}{\mathbb{Q}}
\newcommand{\RR}{\mathbb{R}}

\newcommand{\lcm}{\mathrm{lcm}}
\newcommand{\PSL}{\mathrm{PSL}_2(\RR)}

\usepackage{enumitem}

\newlist{propenum}{enumerate}{1}
\setlist[propenum]{label = (\alph*), font=\upshape, ref = \thetheorem.\mbox{(\alph*)}, leftmargin=2em}
\newlist{lemenum}{enumerate}{1}
\setlist[lemenum]{label = (\alph*), font=\upshape, ref = \thetheorem.\mbox{(\alph*)}, leftmargin=2em}

\newcommand{\conv}{\mathop{\scalebox{1.2}{\raisebox{-0.1ex}{$\ast\,$}}}}%
\makeatletter
\newcommand*{\rom}[1]{\expandafter\@slowromancap\romannumeral #1@}
\makeatother

\newcommand{\Mod}[1]{\ (\mathrm{mod}\ #1)}

\title{Powers of commutators in infinite groups}
\author{Daan Heus\\\\\textit{\small Mathematical Institute, Leiden University}, \href{mailto:daanheus@live.nl}{\small \texttt{daanheus@live.nl}}}
\date{}

\begin{document}

\maketitle

\begin{abstract}
\noindent Given elements $x,u,z$ in a finite group $G$ such that $z$ is the commutator of $x$ and $u$, and the orders of $x$ and $z$ divide respectively integers $k,m \geq 2$, and given an integer $r$ that is coprime to $k$ and $m$, there exists $w \in G$ such that the commutator of $x^r$ and $w$ is conjugate to $z^r$. 
If instead we are given elements $x,y,z \in G$ such that $xy = z$, whose respective orders divide integers $k,l,m \geq 2$, and are given an integer $r$ that is coprime to $k,l$ and $m$, then there exist $x'$, $y'$ and $z'$ conjugate to respectively $x^r$, $y^r$ and $z^r$ such that $x'y' = z'$.
In this paper we completely answer the natural question for which values of $k,l,m,r$ every group has these properties. The proof uses combinatorial group theory and properties of the projective special linear group $\PSL$. 
\end{abstract}

\section{Introduction}

A \textit{Honda} group is a group in which every generator of every subgroup generated by a commutator is itself also a commutator. In algebraic terms we say that a group $G$ is Honda if for every integer $r$ and all $x,u,z \in G$ one has 
$$[x,u] = z, \, \langle z \rangle = \langle z^r \rangle \Rightarrow \exists v,w \in G : [v,w] = z^r.$$
Here we use the convention that $[x,u] = xux^{-1}u^{-1}$. 
In 1953, K. Honda \cite{Honda} proved that every finite group is Honda (Proposition \ref{Proposition: Finite groups qb, qh, h} below). A natural question to ask is if \textit{every} group is Honda.
However, in 1977 (\cite{Pride}, page 488, Result (C)) S.J. Pride gave a family of groups that he showed were not Honda. 
On the other hand, B. Martin proved in \cite{Martin} that many linear algebraic groups are Honda.

We write $\sim$ for ``is conjugate to''. With the same proof that K. Honda used to show that every finite group is Honda, one can also show that every finite group $G$ has the following property: for every integer $r$ and all $x,u,z \in G$ one has
$$[x,u] = z, \, \langle x \rangle = \langle x^r\rangle, \, \langle z \rangle = \langle z^r \rangle \Rightarrow \exists w \in G: [x^r,w] \sim z^r.$$
We call such a group a \textit{quasi-Honda} group.
It is easy to see that every quasi-Honda torsion group is Honda (Proposition \ref{Proposition: torsion}). The groups given by S.J. Pride are quasi-Honda (Proposition \ref{Proposition: Pride groups quasi-Burnside}), so in general quasi-Honda does not imply Honda. One may now wonder if \textit{every} group is quasi-Honda. 

To prove that every finite group is Honda, K. Honda applied an argument that W. Burnside had used in 1911 to prove that every finite group $G$ has the following property:
for every integer $r$ and all $x,y,z \in G$ one has 
$$xy=z, \, \langle x \rangle = \langle x^r \rangle, \, \langle y \rangle = \langle y^r \rangle, \, \langle z \rangle = \langle z^r \rangle \Rightarrow \exists x',y',z' \in G : x'y' = z', \, x' \sim x^r, \, y' \sim y^r, \, z' \sim z^r$$
(\cite{Burnside}, Chapter \rom{15}, Theorem \rom{7}). We will call a group with this property a \textit{quasi-Burnside} group. Since it is easy to prove that every quasi-Burnside group is quasi-Honda (Proposition \ref{Proposition: quasi-Burnside implies quasi-Honda}), 
Honda's theorem is a consequence of Burnside's result.
Just as in the case of the quasi-Honda property one may wonder if every group is quasi-Burnside. 
The answer to both of these questions is negative, but we were able to answer a more refined question, which we proceed to formulate. 

If, in the definitions of quasi-Honda and quasi-Burnside, one has $r = -1$, then one can give explicit formulas for $w$ and for $x', y', z'$; e.g. $w = xu$ and $x' = x^{-1}$, $y' = xy^{-1}x^{-1}$, $z' = z^{-1}$. However, for general $r$ it is only meaningful to ask for such formulas if one guarantees that the conditions $\langle x\rangle = \langle x^r\rangle$, $\langle z\rangle = \langle z^r\rangle$, and (in the quasi-Burnside case) $\langle y\rangle = \langle y^r\rangle$ are satisfied. If $r \notin \{\pm 1\}$, then these conditions are equivalent to the existence of positive integers $k,l,m$ that are coprime to $r$ for which one has $x^k = 1$, ($y^l = 1$,) and $z^m = 1$. This leads to the following definitions. 
Given a ring $R$, we write $R^*$ for its unit group. Let $k,l,m \geq 2$ be integers and let $r \in (\mathbb{Z}/\lcm(k,l,m)\mathbb{Z})^*$. A group $G$ is $(k,l,m,r)$\textit{-quasi-Burnside} if for all $x,y,z \in G$ one has
$$xy = z, \, x^k = y^l = z^m = 1 \Rightarrow \exists x',y',z' \in G: x'y' = z', \, x' \sim x^r, \, y' \sim y^r, z' \sim z^r.$$
Now let $r \in (\mathbb{Z}/\lcm(k,m)\mathbb{Z})^*$. A group $G$ is $(k,m,r)$\textit{-quasi-Honda} if for all $x,u,z \in G$ one has
$$[x,u] = z, \, x^k = z^m = 1 \Rightarrow \exists w \in G: [x^r,w] \sim z^r.$$ 

Clearly, a group $G$ is quasi-Burnside if and only if for all $k,l,m$ and every $r \in (\mathbb{Z}/\lcm(k,l,m)\mathbb{Z})^*$ the group $G$ is $(k,l,m,r)$-quasi-Burnside. 
Analogously, a group $G$ is quasi-Honda if and only if for all $k,m$ and every $r \in (\mathbb{Z}/\lcm(k,m)\mathbb{Z})^*$ the group $G$ is $(k,m,r)$-quasi-Honda. 

In this paper we classify for which values of $k,l,m,r$ respectively $k,m,r$ every group is $(k,l,m,r)$-quasi-Burnside respectively $(k,m,r)$-quasi-Honda. Our results are summarized in the two theorems below.
If $\gcd(k,m) \le 2$, and we are given $r \in (\ZZ/\lcm(k,m)\ZZ)^*$, then we write $r^*$ for the unique element in $(\ZZ/\lcm(k,m)\ZZ)^*$ with $r^* \equiv r \pmod{k}$ and $r^* \equiv -r \pmod{m}$. 
This element $r^*$ exists by Lemma \ref{Lemma: Chinese remainder theorem}. 

\begin{theorem}
\label{introduction theorem qH}
Let $k,m \geq 2$ be integers and let $r \in (\ZZ/\lcm(k,m)\ZZ)^*$. Then the following are equivalent.
\begin{enumerate}[label=\textup{(\arabic*)}, leftmargin=2em]
\item Every group is $(k,m,r)$-quasi-Honda.
\item Every group is $(k,k,m,r)$-quasi-Burnside, or both $\gcd(k,m) \leq 2$ holds and every group is $(k,k,m,r^*)$-quasi-Burnside.
\item We have $\frac{2}{k} + \frac{1}{m} \geq 1$, or $r \in \{\pm 1\}$, or both $\gcd(k,m) \leq 2$ and $r^* \in \{ \pm 1\}$.
\end{enumerate}
\end{theorem}

The equivalence (1) $\Leftrightarrow$ (2) of Theorem \ref{introduction theorem qH} is proved at the end of \S \ref{Section: Passing to a subgroup}, and the equivalence (2) $\Leftrightarrow$ (3) of Theorem \ref{introduction theorem qH} is a direct consequence of Theorem \ref{introduction theorem qB}, which we shall prove first. We write $\PSL$ for the projective special linear group of degree 2 over the real numbers.

\begin{theorem}
\label{introduction theorem qB}
Let $k,l,m \geq 2$ be integers and let $r \in (\ZZ/\lcm(k,l,m)\ZZ)^*$. Then the following are equivalent.
\begin{enumerate}[label=\textup{(\arabic*)}, leftmargin=2em]
\item Every group is $(k,l,m,r)$-quasi-Burnside.
\item The group $\PSL$ is $(k,l,m,r)$-quasi-Burnside.
\item We have $\frac{1}{k} + \frac{1}{l} + \frac{1}{m} \geq 1$ or $r \in \{\pm 1\}$.
\end{enumerate}
\end{theorem}

The implication (1) $\Rightarrow$ (2) of Theorem \ref{introduction theorem qB} is trivial, the implication (2) $\Rightarrow$ (3) is proved in \S \ref{Section: Projective special linear group}, just above Proposition \ref{Proposition: PSL not hemi-Burnside}, and the implication (3) $\Rightarrow$ (1) is proved at the end of \S \ref{Section: Quasi-Burnside and quasi-Honda groups}.

The proof of Theorem \ref{introduction theorem qB} uses facts about products of triples of conjugacy classes in $\PSL$, and the proof of Theorem \ref{introduction theorem qH} uses a special notion of a reduced word in the free product of a group and an infinite cyclic group.

For the entirety of this paper we fix three integers $k,l,m \geq 2$.

In \S \ref{Section: Quasi-Burnside and quasi-Honda groups} we prove some basic facts about $(k,l,m,r)$-quasi-Burnside and $(k,m,r)$-quasi-Honda groups.
We show that every group being $(k,l,m,r)$-quasi-Burnside is equivalent to the existence of $g,h \in B_{k,l,m} := \langle a,c \mid a^k = (a^{-1}c)^l = c^m = 1 \rangle$ such that $a^r \cdot g(a^{-1}c)^rg^{-1} = hc^rh^{-1}$ (Proposition \ref{Proposition: universal q-Burnside group}), and we show that every group being $(k,m,r)$-quasi-Honda is equivalent to the existence of $d,e \in H_{k,m}:= \langle a,b,c \mid [a,b] = c, \, a^k = c^m = 1 \rangle$ such that $[a^r,d] = ec^re^{-1}$ (Proposition \ref{Proposition: universal q-Honda group}). Such elements $g,h$ and $d,e$, if they exist, are essentially the formulas that we asked for.
The group $B_{k,l,m}$ is known as a \textit{von Dyck group} or a \textit{triangle group} in the literature. The von Dyck groups are usually studied via the group $\PSL$, so it is not surprising that $\PSL$ plays an important role in this paper. 

As stated in Proposition \ref{Proposition: r st G is klmr qB qH is subgroup}, one readily checks that the elements $r$ in $(\ZZ/\lcm(k,m)\ZZ)^*$ respectively $(\ZZ/\lcm(k,l,m)\ZZ)^*$ for which a given group $G$ is $(k,m,r)$-quasi-Honda respectively $(k,l,m,r)$-quasi-Burnside form a subgroup. In the quasi-Burnside case, we denote that subgroup of $(\ZZ/\lcm(k,l,m)\ZZ)^*$ by $M_G$.
In \S \ref{Section: Projective special linear group} we determine $M_G$ for $G = \PSL$ (Proposition \ref{Proposition: If r in Mklm, then r in pm 1}), allowing us to prove the implication (2) $\Rightarrow$ (3) of Theorem \ref{introduction theorem qB}. Then, in \S \ref{Section: universal groups}, we examine the structure of the groups $H_{k,m}, B_{k,l,m}$ and show that $B_{k,k,m}$ may be viewed as a subgroup of $H_{k,m}$ (Proposition \ref{Proposition: embedding}), which is the key to the proof of the equivalence (1) $\Leftrightarrow$ (2) of Theorem \ref{introduction theorem qH}. 
In \S \ref{Section: Reduced words} we define the special notion of a reduced word mentioned above, and prove some of its properties. In \S \ref{Section: Passing to a subgroup} this notion enables us to prove Proposition \ref{Proposition: elliminating b v2}, which we then use, together with the theory from \S \ref{Section: universal groups}, to prove the equivalence (1) $\Leftrightarrow$ (2) of Theorem \ref{introduction theorem qH}.

\section{Quasi-Burnside and quasi-Honda groups}
\label{Section: Quasi-Burnside and quasi-Honda groups}

As mentioned in the introduction, we fix integers $k,l,m \geq 2$ for the entirety of this paper. 

In this section we prove some basic facts about $(k,l,m,r)$-quasi-Burnside and $(k,m,r)$-quasi-Honda groups, and at the end we prove the implication (3) $\Rightarrow$ (1) of Theorem \ref{introduction theorem qB}.

Let $x,y$ be two elements of a group $G$. The notation $x \sim y$ means ``$x$ is conjugate to $y$''. We write $\mathrm{ord}(x)$ for the order of $x$, and we write ${}^y x$ for $yxy^{-1}$. Given an integer $n$, we have $(yxy^{-1})^n = yx^ny^{-1}$, hence the notation ${}^y x^n$ is unambiguous. 

\begin{proposition}
Let $G$ be a group, let $x,y,z,u \in G$, and let $r$ be an integer.
\begin{propenum}
    \item \label{Proposition: x or z infinite order quasi-Honda}
    Suppose $\mathrm{ord}(x) = \infty$ or $\mathrm{ord}(z)= \infty$. If $[x,u] = z$, $ \langle x \rangle = \langle x^r\rangle$, $ \langle z \rangle = \langle z^r \rangle$, then there exists $w \in G$ such that $[x^r,w] = z^r$.
    \item \label{Proposition: x or y or z infinite order quasi-Brunside}
    Suppose $\mathrm{ord}(x) = \infty$ or $\mathrm{ord}(y)= \infty$ or $\mathrm{ord}(z)= \infty$. If $xy = z$, $ \langle x \rangle = \langle x^r\rangle$, $ \langle y \rangle = \langle y^r \rangle$, $ \langle z \rangle = \langle z^r \rangle$, then there exist $x'$, $y'$, $z' \in G$ such that $x'y' = z'$, $ x' \sim x^r$, $ y' \sim y^r$, $ z' \sim z^r$.
\end{propenum}  
\end{proposition}

\begin{proof}
The proofs are analogous, so we only prove (a).
Suppose $[x,u] = z$, $\langle x \rangle = \langle x^r \rangle$, $\langle z \rangle = \langle z^r \rangle$. If $\mathrm{ord}(x) = \infty$, then $r \in \{\pm 1\}$ since $\langle x \rangle = \langle x^r \rangle$. Analogously we have $r \in \{\pm1\}$ if $\mathrm{ord}(z) = \infty$. As shown in the introduction, if $r \in \{\pm 1\}$, then we can give an explicit formula for $w$ as above.
\end{proof}

\begin{lemma}
\label{Lemma: 5}
Let $r$ and $n$ be coprime integers, and let $S$ be a finite set of prime numbers. Then there exists a positive integer $r'$ with $r' \equiv r \pmod{n}$, that is not divisible by any element of $S$.
\end{lemma}

\begin{proof}
Define $T:= \{p \in S: p \nmid r\}$ and $q:= \prod_{p \in T}p$. Then $r'= r+nq$ will do.
\end{proof}

\begin{proposition}
\label{Proposition: torsion}
If a group is quasi-Honda and torsion, then it is Honda.
\end{proposition}

\begin{proof}
Let $G$ be a group that is quasi-Honda and torsion. Let $r$ be an integer, $x,u,z \in G$ with $[x,u]=z, \, \langle z \rangle = \langle z^r \rangle$. Let $S$ be the set of prime numbers that divide $\mathrm{ord}(x) \mathrm{ord}(z)$. We have $\mathrm{gcd}(r,\mathrm{ord}(z)) = 1$, so by Lemma \ref{Lemma: 5} there exists a positive integer $r'$ with $r' \equiv r \pmod{\mathrm{ord}(z)}$, that is coprime to $\mathrm{ord}(x) \mathrm{ord}(z)$. Choose such an $r'$. Then $\langle z \rangle = \langle z^{r'} \rangle$ and $\langle x \rangle = \langle x^{r'} \rangle$, so there exists $w \in G$ such that $[x^{r'},w] \sim z^{r'} = z^r$. Thus $z^r$ is a commutator.
\end{proof}

\begin{proposition}
\begin{propenum}
\item \label{Proposition: klmr quasi-Burnside implies kmr quasi-Burnside}
Let $r \in (\mathbb{Z}/\lcm(k,m)\mathbb{Z})^*$. If a group is $(k,k,m,r)$-quasi-Burnside, then it is $(k,m,r)$-quasi-Honda.
\item \label{Proposition: quasi-Burnside implies quasi-Honda}
If a group is quasi-Burnside, then it is quasi-Honda.
\end{propenum}
\end{proposition}

\begin{proof}
It suffices to only prove (a), since a group is quasi-Burnside respectively quasi-Honda if and only if it is $(k,k,m,r)$-quasi-Burnside respectively $(k,m,r)$-quasi-Honda for all $k,m,r$. 
Let $G$ be a group. Given $x,z \in G$, we have 
\begin{align*}
    \exists w \in G: [x^r,w] \sim z^r & \Leftrightarrow \exists w \in G: x^r \cdot {}^wx^{-r} \sim z^r \\
    & \Leftrightarrow \exists a,b,w \in G: {}^ax^r \cdot {}^{aw}x^{-r} = {}^bz^r \\
    & \Leftrightarrow \exists x',y',z' \in G: x'y' = z',\, x' \sim x^r,\, y' \sim x^{-r},\, z' \sim z^r.
\end{align*}
Thus $G$ is $(k,m,r)$-quasi-Honda if and only if for all $x,u,z \in G$ we have that if $[x,u] = z$, $x^k=z^m=1$, then there exist $x',y',z' \in G$ such that $x'y'=z'$, $x' \sim x^r$, $y' \sim x^{-r}$, $z' \sim z^{-r}$. This is the definition of a $(k,k,m,r)$-quasi-Burnside group restricted to the case where $y = {}^ux^{-1}$ for some $u \in G$.
\end{proof}

\begin{proposition}
\label{Proposition: Finite groups qb, qh, h}
Every finite group is quasi-Burnside, quasi-Honda and Honda.
\end{proposition}

\begin{proof}
Given a conjugacy class $C$ of a group, and an integer $s$, we write $C^s$ for the conjugacy class $\{c^s:c \in C\}$. We write $\#$ for the cardinality of a set. Theorem \rom{7} from Chapter \rom{15} of \cite{Burnside} by Burnside states the following. Let $G$ be a finite group and let $s$ be an integer that is coprime to the exponent of $G$. Let $C,D,E$ be conjugacy classes of $G$. Then for all $z \in E$ and all $z' \in E^s$ one has
$$\#\{(x,y) \in C \times D: xy = z\} = \#\{(x',y') \in C^s \times D^s: x'y' = z'\}.$$
Now let $G$ be a finite group, $r$ an integer. Let $x,y,z \in G$ such that $xy=z$, $\langle x \rangle = \langle x^r \rangle$, $\langle y \rangle = \langle y^r \rangle$, $\langle z \rangle = \langle z^r \rangle$ and let $n:= \mathrm{ord}(x)\mathrm{ord}(u)\mathrm{ord}(z)$. Notice that $\mathrm{gcd}(r,n) = 1$. Let $S$ be the set of prime numbers that divide the exponent of $G$. By Lemma \ref{Lemma: 5} there exists a positive integer $r'$ that is coprime to the exponent of $G$ and that satisfies $r' \equiv r \pmod{n}$. Choose such $r'$. Applying Burnside's Theorem \rom{7} with $C,D,E$ equal to the conjugacy classes of respectively $x,y,z$, and with $s=r'$, we find that there exist $x',y',z' \in G$ such that $x'y' = z'$, $x'\sim x^{r'} = x^r$, $y'\sim y^{r'} = y^r$, $z'\sim z^{r'} = z^r$. Thus every group is quasi-Burnside. By Proposition \ref{Proposition: quasi-Burnside implies quasi-Honda} every finite group is quasi-Honda, and then by Proposition \ref{Proposition: torsion} every finite group is Honda.
\end{proof}

Given two groups $G,H$, we denote by $G \conv H$ the free product of $G$ and $H$. Result (C) on page 488 of \cite{Pride} states that given an integer $n \geq 2$, the group $P_n:= \langle a,b,c \mid [a,b]=c, \, c^n = 1\rangle$ is not Honda. It turns out that $P_n$ is quasi-Burnside and quasi-Honda.

\begin{proposition}
\label{Proposition: Pride groups quasi-Burnside}
    Let $n \geq 2$ be an integer. Then $P_n$ is quasi-Burnside and quasi-Honda.
\end{proposition}

\begin{proof}
Let $x,y,z \in P_n$ and let $r$ be an integer such that $xy = z$, $\langle x \rangle = \langle x^r\rangle$, $\langle y \rangle = \langle y^r\rangle$, $\langle z \rangle = \langle z^r \rangle$. We will show that there exist $x', y', z' \in P_n$ with $x'y' = z'$, $x' \sim x^r$, $y' \sim y^r$, $z' \sim z^r$. 
By Proposition \ref{Proposition: x or y or z infinite order quasi-Brunside} we may assume that $x,y,z$ have finite order. 
Let $\langle a \rangle$, $\langle b \rangle$ be infinite cyclic groups, let $A_n = \langle c_i \, (i \in \mathbb{Z})\mid c_i^n = 1 \, (i \in \mathbb{Z}) \rangle \conv \langle b \rangle$, and define $A_n \rtimes \langle a \rangle$ by $ac_ia^{-1} = c_{i+1}$ for all $i$ and $ aba^{-1} = c_0b$.
One readily shows that there is an isomorphism $A_n \rtimes \langle a \rangle \xrightarrow{\sim} P_n: a \mapsto a$, $b \mapsto b$, $c_i \mapsto a^ica^{-i}$.
Any element of finite order in $A_n \rtimes \langle a \rangle$ belongs to the kernel of the canonical quotient map $A_n \rtimes \langle a \rangle \to (A_n \rtimes \langle a \rangle)/A_n \cong \mathbb{Z}$, so to $A_n$. 
Note that $A_n$ is the free product of $\langle b \rangle$ and a countably infinite number of cyclic groups of order $n$, generated by the $c_i$'s.
Corollary 1 of Proposition 2 in Section \rom{1}.1.3 of \cite{Serre} states that every element of finite order in a free product is conjugate to an element of one of the groups in the free product. 
Thus any element of finite order in $A_n \rtimes \langle a \rangle$ is conjugate in $A_n$ to a power of $c_i$ for some $i$. 
View $x,y,z$ as elements of $A_n \rtimes \langle a \rangle$ and let $s,t,u$ be integers such that $x,y,z$ are conjugate in $A_n$ to respectively $c_h^s,c_i^t,c_j^u$ for some integers $h,i,j$.
Let $\langle \beta \rangle, \langle \gamma \rangle$ be cyclic groups of respective orders $\infty, n$, and let $\sigma: A_n \to \langle \beta \rangle \times \langle \gamma \rangle$ be the homomorphism that sends $b$ to $\beta$ and sends $c_i$ to $\gamma$ for all $i$. Since $\langle \beta \rangle \times \langle \gamma \rangle$ is abelian, we have $\gamma^u = \sigma(z) = \sigma(x)\sigma(y) = \gamma^{s+t}$. Thus $u \equiv s+t \pmod{n}$, so $c^{sr} \cdot c^{tr} = c^{(s+t)r} = c^{ur}$. In $P_n$ we have $c^{sr} \sim x^r$, $c^{tr} \sim y^r$, $c^{ur} \sim z^r$, so $P_n$ is quasi-Burnside.

Now Proposition \ref{Proposition: quasi-Burnside implies quasi-Honda} gives that $P_n$ is quasi-Honda.
\end{proof}

\begin{proposition}
\label{Proposition: r st G is klmr qB qH is subgroup}
Let $G$ be a group. Then $\{r \in (\ZZ/\lcm(k,m)\ZZ)^* \mid G \text{ is } (k,m,r) \text{-quasi-Honda}\}$ and $\{r \in (\ZZ/\lcm(k,l,m)\ZZ)^* \mid G \text{ is } (k,l,m,r) \text{-quasi-Burnside}\}$ are subgroups of $(\ZZ/\lcm(k,m)\ZZ)^*$ respectively $(\ZZ/\lcm(k,l,m)\ZZ)^*$ containing $-1$. \hfill \qedsymbol
\end{proposition}

Recall that $k,l,m\geq 2$ are still fixed integers. 

Let $G$ be a group and let $n$ be an integer. We write $G[n]$ for $\{g \in G: g^n = 1\}$ and $G[n]^0$ for $G[n] \setminus \{1\}$. We write $G/{\sim}$ for the set of conjugacy classes of $G$. 
Given $C,D,E \in G/{\sim}$, we write $CDE$ for $\{cde : c \in C, d \in D, e \in E\} \subseteq G$ and write $C^n$ for $\{c^n: c \in C\} \in G/{\sim}$. 
Define 
$$B_G:= \{(C,D,E) \in (G/{\sim})^3: 1 \in CDE, \, C \subseteq G[k]^0, \, D \subseteq G[l]^0, \, E \subseteq G[m]^0\}.$$
Note that given a representative $[r] \in \ZZ$ of some $r \in (\ZZ/\lcm(k,l,m)\ZZ)^*$ and $C \subseteq G[k]$, the conjugacy class $C^r := C^{[r]}$ is well-defined because $C^{[r]} = C^{[r]+\lcm(k,l,m)}$. Given $D \subseteq G[l]$ and $E \subseteq G[m]$, the same is true for $D^r$ and $E^r$. Finally we define 
$$M_G:= \{r \in (\ZZ/\lcm(k,l,m)\ZZ)^*\mid \forall (C,D,E) \in B_G: (C^r,D^r,E^r) \in B_G\}.$$

In Lemma \ref{Lemma: multiplier general}, we use the following equivalent definition of a $(k,l,m,r)$-quasi-Burnside group. 
Let $r \in (\ZZ/\lcm(k,l,m)\ZZ)^*$.
A group $G$ is $(k,l,m,r)$-quasi-Burnside if for all $c,d,e \in G$ one has
$$cde=c^k = d^l = e^m = 1 \Rightarrow \exists c',d',e' \in G: c'd'e'=1, \, c' \sim c^r , \,d' \sim d^r, \,e' \sim e^r.$$

\begin{lemma}
\label{Lemma: multiplier general}
Let $G$ be a group and let $r \in (\ZZ/\lcm(k,l,m)\ZZ)^*$. Then $G$ is $(k,l,m,r)$-quasi-Burnside if and only if $r \in M_G$. Moreover, $M_G$ is a subgroup of $(\ZZ/\lcm(k,l,m)\ZZ)^*$ containing $-1$.
\end{lemma}

\begin{proof}
The group $G$ is $(k,l,m,r)$ quasi-Burnside if and only if for all $c,d,e \in G$ with $c^k = d^l = e^m = cde = 1$, there exist $c',d',e' \in G$ such that $c'd'e' = 1, c' \sim c^r, \, d' \sim d^r, \, e' \sim e^r$, so if and only if for all $(C,D,E) \in (G/{\sim})^3$ with $C \subseteq G[k], \, D \subseteq G[l], \, E \subseteq G[m], \, 1 \in CDE$ we have $1 \in C^rD^rE^r$. Clearly, if $(C,D,E) \in (G/{\sim})^3$ with $C \subseteq G[k], \, D \subseteq G[l], \, E \subseteq G[m], \, 1 \in CDE$ and at least one of $C,D,E$ is equal to $\{1\}$, then $1 \in C^rD^rE^r$. Thus every group is $(k,l,m,r)$ quasi-Burnside if and only if for all $(C,D,E) \in (G/{\sim})^3$ with $C \subseteq G[k]^0, \, D \subseteq G[l]^0, \, E \subseteq G[m]^0, \, 1 \in CDE$ we have $1 \in C^rD^rE^r$, which is equivalent to $r \in M_G$.
The second statement then follows from Proposition \ref{Proposition: r st G is klmr qB qH is subgroup}.
\end{proof}

\begin{lemma}
\label{Lemma: Chinese remainder theorem}
The map $(\ZZ/\lcm(k,m)\ZZ)^* \to (\ZZ/k\ZZ)^* \times (\ZZ/m\ZZ)^*$, $a \mapsto (a \Mod{k}, a \Mod{m})$ is an isomorphism of groups if and only if $\gcd(k,m) \le 2$.
\end{lemma}

\begin{proof}
Denote by $f$ the map from the lemma and by $\varphi$ the Euler totient function. 
Clearly $f$ is an injective group homomorphism.
Since $\varphi(\lcm(k,m)) \cdot \varphi(\gcd(k,m)) = \varphi(k) \cdot \varphi(m)$, we have that $f$ is surjective if and only if $\varphi(\gcd(k,m)) = 1$, which is equivalent to $\gcd(k,m) \le 2$.
\end{proof}

If $\gcd(k,m) \le 2$ and $r \in (\ZZ/\lcm(k,m)\ZZ)^*$, then write $r^*$ for the unique element of $(\ZZ/\lcm(k,m)\ZZ)^*$ with $r^* \equiv r \pmod{k}$ and $r^* \equiv -r \pmod{m}$. Clearly $r^* = 1^* \cdot r$ and $(r^*)^* = r$. Also note that $(-r)^* = -(r^*)$, so the notation $-r^*$ is unambiguous.

\begin{proposition}
\label{Proposition: r^*}
Let $r \in \{\pm 1\} \subseteq (\ZZ/\lcm(k,m)\ZZ)^*$. Then every group is $(k,m,r)$-quasi-Honda. Moreover, if $\gcd(k,m) \le 2$, then every group is $(k,m,r^*)$-quasi-Honda.
\end{proposition}

\begin{proof}
The first statement was explicitly proved in the introduction.
Suppose $\gcd(k,m) \le 2$. If for elements $x,u,z$ of some group $G$ we have $[x,u] = z$, $x^k = z^m = 1$, then $[x^{1^*}, u^{-1}] = [x,u^{-1}] = {}^{u^{-1}}z^{-1} = {}^{u^{-1}}z^{1^*}$. Proposition \ref{Proposition: r st G is klmr qB qH is subgroup} then gives that $G$ is $(k,m,-1^*)$-quasi-Honda.
\end{proof}

Define $B_{k,l,m} = \langle a,c \mid a^k=(a^{-1}c)^l=c^m=1 \rangle$ and $H_{k,m} = \langle a,b,c \mid [a,b] = c,\, a^k = c^m = 1 \rangle$. These groups have certain universal properties, which are described in Proposition \ref{Proposition: universal groups}. In \S \ref{Section: universal groups}, we will determine the structure of these groups, and prove some of their properties.

\begin{proposition}
\label{Proposition: universal groups}
Let $r \in (\mathbb{Z}/\lcm(k,l,m)\mathbb{Z})^*$.
\begin{propenum}
    \item \label{Proposition: universal q-Burnside group}
    Every group is $(k,l,m,r)$-quasi-Burnside if and only if there exist $g,h\in B_{k,l,m}$ such that $a^r \cdot {}^g(a^{-1}c)^r = {}^hc^r$.
    \item \label{Proposition: universal q-Honda group}
    Every group is $(k,m,r)$-quasi-Honda if and only if there exist $d,e \in H_{k,m}$ such that $[a^r,d] = {}^ec^r$.
\end{propenum}
\end{proposition}

\begin{proof}
The proofs are similar, so we only prove (a).
``$\Rightarrow$'' Trivial.
``$\Leftarrow$'' Let $g,h \in B_{k,l,m}$ such that $a^r \cdot {}^g(a^{-1}c)^r = {}^hc^r$. 
If $x,y,z$ are elements of a group $G$ with $xy = z$, $x^k = y^l = z^m = 1$, then there exists a homomorphism $f:B_{k,l,m} \to G: a \mapsto x$, $c \mapsto z$. Applying $f$ to $a^r \cdot {}^g(a^{-1}c)^r = {}^hc^r$ gives ${}^{f(h)}z^r = x^r \cdot {}^{f(g)}(x^{-1}z)^r = x^r \cdot {}^{f(g)}y^r$.
\end{proof}

\begin{proof}[Proof of the implication (3) $\Rightarrow$ (1) of Theorem \ref{introduction theorem qB}]
As shown in the introduction, if $r \in \{\pm 1\}$, then every group is $(k,l,m,r)$-quasi-Burnside.
If $\frac{1}{k}+\frac{1}{l}+\frac{1}{m} = 1$, then up to permutation we have that $(k,l,m) \in \{(2,3,6),(2,4,4),(3,3,3)\}$, thus $(\ZZ/\lcm(k,l,m)\ZZ)^* = \{\pm1\}$, so every group is $(k,l,m,r)$-quasi-Burnside.

Suppose $\frac{1}{k}+\frac{1}{l}+\frac{1}{m} > 1$. Then it is a well-known fact that the von Dyck group $B_{k,l,m}$ is finite. (See for example \S 15 of \cite{Johnson}.) By Proposition \ref{Proposition: Finite groups qb, qh, h}, every finite group is $(k,l,m,r)$-quasi-Burnside, so there exist $f,g,h \in B_{k,l,m}$ such that ${}^fa^r \cdot {}^g(a^{-1}c^{-1})^r \cdot {}^hc^r = 1$. For such $f,g,h$ we have $a^r \cdot {}^{f^{-1}g}(a^{-1}c^{-1})^r \cdot {}^{f^{-1}h}c^r = 1$, so Proposition \ref{Proposition: universal q-Burnside group} gives that every group is $(k,l,m,r)$-quasi-Burnside.
\end{proof}

\section{Projective special linear group}
\label{Section: Projective special linear group}

The \textit{special linear group of degree 2 over the real numbers} $\mathrm{SL}_2(\RR)$ is the group of real 2 by 2 matrices with determinant 1. Denote by $\mathrm{I}$ the identity element of $\mathrm{SL}_2(\mathbb{R})$. The \textit{projective special linear group of degree 2 over the real numbers} $\PSL$ is the quotient group $\mathrm{SL}_2(\RR)/\langle -\mathrm{I} \rangle$.

In this section, we describe $B_G$ and $M_G$, as defined just before Lemma \ref{Lemma: multiplier general}, for $G = \PSL$. This allows us to prove the implication (2) $\Rightarrow$ (3) of Theorem \ref{introduction theorem qB} (just above Proposition \ref{Proposition: PSL not hemi-Burnside}), and Proposition \ref{Proposition: PSL not hemi-Burnside}, which we use in \S \ref{Section: universal groups} towards the proof of the equivalence (1) $\Leftrightarrow$ (2) of Theorem \ref{introduction theorem qH}.

Let $H = (\frac{1}{k}\ZZ/\ZZ) \oplus (\frac{1}{l}\ZZ/\ZZ) \oplus (\frac{1}{m}\ZZ/\ZZ)$ and regard it as a subgroup of $(\RR/\ZZ)^3$. 
If $r \in \ZZ/\lcm(k,l,m)\ZZ$ and $(\frac{a}{k},\frac{b}{l},\frac{c}{m}) \in H$, then we define $r \cdot (\frac{a}{k},\frac{b}{l},\frac{c}{m}) = (\frac{ra}{k},\frac{rb}{l},\frac{rc}{m})$, making $H$ into a $\ZZ/\lcm(k,l,m)\ZZ$-module.
Let $S = \{(x,y,z) \in \RR^3 : x>0$, $y>0$, $z>0$, $x+y+z<1\}$ and view it as a subset of $(\RR/\ZZ)^3$ via the canonical quotient map $\RR^3 \to (\RR/\ZZ)^3$.
Define 
$$M = \{r \in (\ZZ/\lcm(k,l,m)\ZZ)^*: r \cdot (H \cap ( S \cup -S)) = H \cap(S \cup -S)\}.$$ 
See below a picture of $S$ and $-S$ inside $(\RR/\ZZ)^3$.

\begin{center}
    \includegraphics[scale = 0.2]{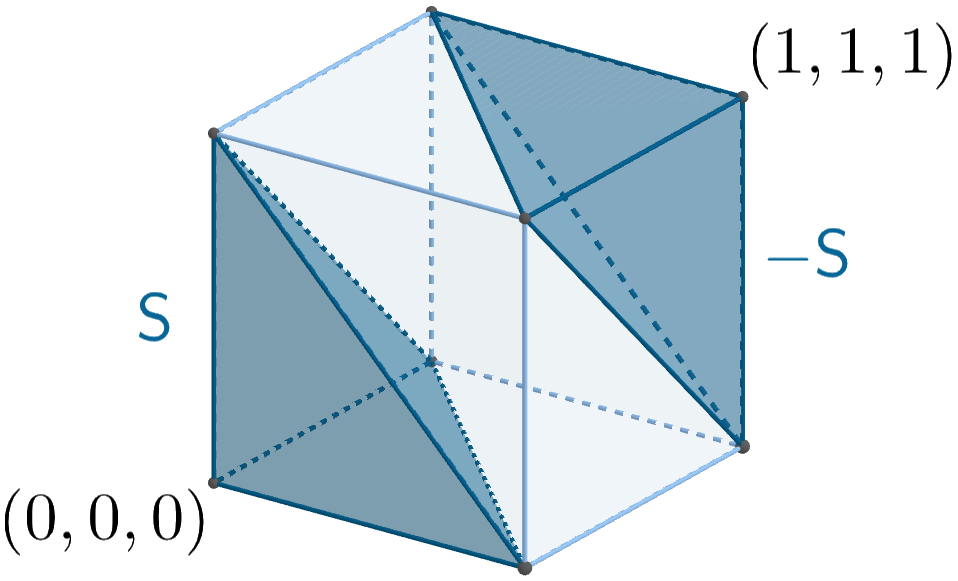}
\end{center}

\begin{proposition}
\begin{propenum}
\item \label{Proposition: klmr-Burnside if and only if r in Mklm}
We have $M_{\PSL} = M$. 
\item \label{Proposition: If r in Mklm, then r in pm 1}
If $\frac{1}{k}+\frac{1}{l}+\frac{1}{m}<1$, then $M_{\PSL} = \{\pm1\}$.
\end{propenum}
\end{proposition}

We prove Proposition \ref{Proposition: klmr-Burnside if and only if r in Mklm} just below Lemma \ref{Lemma: BPSL2R}, and prove Proposition \ref{Proposition: If r in Mklm, then r in pm 1} just below Lemma \ref{Lemma: 1<c<m-2}.

Given $a \in \RR/\ZZ$, we write $\sigma_a$ for $\begin{pmatrix}
    \cos [a]\pi & -\sin [a]\pi \\
    \sin [a]\pi & \cos [a]\pi
\end{pmatrix} \cdot \langle -\mathrm{I} \rangle \in \PSL$, where $[a] \in \RR$ is some representative of $a$, and we write $C_a$ for the conjugacy class of $\sigma_a$. One readily shows that the definitions of $\sigma_a$ and $C_a$ do not depend on the choice of representative.
Write $\PSL^{\mathrm{Tor}}$ for the set of elements of finite order in $\PSL$, and given a subset $F \subseteq \PSL$ that is stable under conjugation in $\PSL$, denote by $F/{\sim}$ the quotient set of $F$ by conjugation in $\PSL$.
There is a $(\ZZ/\lcm(k,l,m)\ZZ)^*$-action on the set $X = (\PSL[k]/{\sim}) \times (\PSL[l]/{\sim}) \times (\PSL[m]/{\sim})$, defined by mapping $(r,(C_a,C_b,C_c)) \in (\ZZ/\lcm(k,l,m)\ZZ)^* \times X$ to $(C_a^r,C_b^r,C_c^r) \in X$. Write $*$ for this action, and note that for such $r$ and $(C_a,C_b,C_c)$ we have $r*(C_a,C_b,C_c) = (C_a^r,C_b^r,C_c^r) = (C_{ra},C_{rb},C_{rc})$.

\begin{lemma}
\begin{lemenum}
\item \label{Lemma: PSL TOR}
The map $\QQ/\ZZ \to \PSL^\mathrm{Tor}/{\sim}$ sending $a$ to $C_a$ is a bijection. Moreover, if $n$ is an integer and $a \in \QQ/\ZZ$, then $n \cdot a = 0$ if and only if $C_a^n = \{1\}$.
\item \label{Lemma: bijection H, PSLk/sim}
The map $\psi:H \to \PSL[k]/{\sim} \times \PSL[l]/{\sim} \times \PSL[m]/{\sim}$ that sends $(a,b,c)$ to $(C_a,C_b,C_c)$ is a bijection that respects the $(\ZZ/\lcm(k,l,m)\ZZ)^*$-actions, i.e., for all $(a,b,c) \in H$ and all $r \in (\ZZ/\lcm(k,l,m)\ZZ)^*$ we have $r*\psi(a,b,c) = \psi(r(a,b,c))$.
\end{lemenum}
\end{lemma}

\begin{proof}
We first prove statement (a). Write $\varphi$ for the map from statement (a), and note that $\mathrm{im}(\varphi) \subseteq \PSL^\mathrm{Tor}/{\sim}$, so it is well-defined. 
Given $c \in \RR$, we define $s_c:=\begin{pmatrix}
    \cos c\pi & -\sin c\pi \\
    \sin c\pi & \cos c\pi
\end{pmatrix}$. 
It is a well-known fact that if $A \in \mathrm{SL}_2(\mathbb{R})$ is of finite order, then for some $c \in \QQ$ we have $A \sim s_c$.
The pre-image under the canonical map $\mathrm{SL}_2(\mathbb{R}) \to \mathrm{PSL}_2(\mathbb{R})$ of a torsion element is a torsion element, so the elements of finite order in $\mathrm{PSL}_2(\mathbb{R})$ are conjugate to some $\sigma_a$ with $a \in \QQ/\ZZ$.
Thus the map from the lemma is surjective.

To show that $\varphi$ is injective, suppose $C_a = C_b$ with $a,b \in \QQ/\ZZ$. If $a=0$ or $b=0$, then $a = 0 = b$, so we may assume $a,b$ to be nonzero. Let $[a],[b] \in (0,1) \cap \QQ$ be representatives of respectively $a,b$. Then $s_{[a]} \sim s_{[b]}$ or $s_{[a]} \sim -s_{[b]} = s_{-1+[b]}$. It is a well-known fact that if for an element $A\in \mathrm{SL}_2(\RR)$ we have $|\mathrm{tr}(A)|<2$, then $A$ is conjugate to a unique matrix of the form $s_c$ with $c \in (-1,1) \setminus \{0\}$. (See for example Theorem 3.1 of \cite{Conrad}.) 
For all $c \in ((-1,1)\setminus \{0\}) \cap \QQ$ we have $|\mathrm{tr}(s_c)|<2$, so $s_{[a]} = s_{[b]}$ or $s_{[a]} = s_{-1+[b]}$. Clearly $s_{[a]} = s_{-1+[b]}$ is not possible, so $s_{[a]} = s_{[b]}$, and thus $a=b$. This proves the first part of statement (a). One easily shows that the second part follows from the first.

Now we prove statement (b). By statement (a), we have that $\psi$ restricted to ${\frac{1}{k}\ZZ/\ZZ}$ bijectively maps to $\PSL[k]/{\sim}$. The same is true when we replace $k$ by $l$ or $m$, so $\psi$ is a bijection. This bijection respects the $(\ZZ/\lcm(k,l,m)\ZZ)^*$-actions, since for all $(a,b,c) \in H$ and $r \in (\ZZ/\lcm(k,l,m)\ZZ)^*$ we have $(C_a^r,C_b^r,C_c^r) = (C_{ra},C_{rb},C_{rc})$.
\end{proof}

\begin{lemma}
\label{Lemma: Orevkov}
Let $a,b,c \in (\RR/\ZZ)\setminus \{0\}$ and let $[a],[b],[c] \in (0,1)$ be their respective representatives. Then $1 \in C_a C_b C_c$ if and only if $[a]+[b]+[c]\notin(1,2)$.
\end{lemma}

\begin{proof}
This is a reformulation of a part of Corollary 3.3 (b) of \cite{Orevkov}.
\end{proof}

In \S 8 of \cite{Thesis}, Lemma \ref{Lemma: Orevkov} was deduced from a correspondence between solutions $(x,u,z)\in C_a \times C_b \times C_c$ to $xuz=1$ and hyperbolic triangles with angles $[a]\pi$, $[b]\pi$, $[c]\pi$, where $[a], [b], [c] \in (0,1)$ are respective representatives of $a,b,c$, and the well-known fact that there exists a hyperbolic triangle with angles $[a]\pi$, $[b]\pi$, $[c]\pi$ if and only if $[a] + [b] + [c] < 1$. 

\begin{lemma}
\label{Lemma: BPSL2R}
We have $B_{\PSL} = \{(C_a,C_b,C_c) \in (\PSL^{\mathrm{Tor}}/{\sim})^3 \mid 1 \in C_aC_bC_c, \, C_a^k = C_b^l = C_c^m = \{1\}, \, a,b,c \in (\QQ/\ZZ)\setminus \{0\}\}$. Moreover, let $T = \{(a,b,c) \in ((\QQ/\ZZ)\setminus \{0\})^3 \mid a + b + c = 0\}$. Then there is a bijection $B_{\PSL} \to H \cap (S \cup - S \cup T)$ sending $(C_a,C_b,C_c)$ to $(a,b,c)$.
\end{lemma}

\begin{proof}
Let $G = \PSL$. By Lemma \ref{Lemma: PSL TOR} we have that the group $B_G$ consists of the elements $(C_a,C_b,C_c)$ with $1 \in C_aC_bC_c, \, C_a^k = C_b^l = C_c^m = \{1\}, \, a,b,c \in (\QQ/\ZZ)\setminus \{0\}$. By Lemma \ref{Lemma: Orevkov}, these elements of $B_G$ are exactly the elements $(C_a,C_b,C_c)  \in (G^{\mathrm{Tor}}/{\sim})^3$ with $(a,b,c) \in H \cap (S \cup -S \cup T)$. Hence, the map from the lemma exists and is a bijection.
\end{proof}

\begin{proof}[Proof of Proposition \ref{Proposition: klmr-Burnside if and only if r in Mklm}]
By definition of $M_{\PSL}$ we have that $r \in M_{\PSL}$ is equivalent to $r * B_{\PSL} \subseteq B_{\PSL}$. By Lemmas \ref{Lemma: bijection H, PSLk/sim} and \ref{Lemma: BPSL2R}, this is equivalent to $r \cdot (H \cap (S \cup -S \cup T)) \subseteq H \cap (S \cup -S \cup T)$.
If $(a,b,c) \in H$, then clearly $a+b+c = 0$ if and only if $r(a+b+c) = 0$, so multiplication by $r$ is a permutation of $H \cap T$. So $r \cdot (H \cap (S \cup -S \cup T)) \subseteq H \cap (S \cup -S \cup T)$ is equivalent to $r(H \cap (S \cup -S)) \subseteq H \cap(S \cup -S)$, which is equivalent to $r \in M$ since multiplication by $r$ on $H\cap(S \cup -S)$ is injective and hence bijective.
\end{proof}

We will now prove some lemmas so we can find $M_{\PSL}$ in Proposition \ref{Proposition: If r in Mklm, then r in pm 1}.
Recall that we defined $S = \{(x,y,z) \in \RR^3 : x>0$, $y>0$, $z>0$, $x+y+z<1\}$, and that we view it as a subset of $(\RR/\ZZ)^3$.

\begin{lemma}
\label{lemma: ceiling}
Let $r \in (\ZZ/\lcm(k,l,m)\ZZ)^*$ and let $a,b$ be integers such that $0<a<k$, $0<b<l$, $\frac{a}{k} + \frac{b}{l}<1$. Then $\# \{z \in \frac{1}{m}\ZZ/\ZZ \mid (\frac{a}{k},\frac{b}{l},z) \in S\} = \lceil m - \frac{am}{k} - \frac{bm}{k} \rceil - 1.$
\end{lemma}

\begin{proof}
Given $z \in \frac{1}{m}\ZZ/\ZZ$ such that $(\frac{a}{k},\frac{b}{l},z) \in S$, its representative $[z] \in (0,1)$ satisfies $0 < [z] < 1- \frac{a}{k} - \frac{b}{l}$, and there is a unique integer $0 < c < m - \frac{am}{k} - \frac{bm}{k}$ with $[z] = \frac{c}{m}$. Given such $c$, the class $z$ of $\frac{c}{m}$ in $\frac{1}{m}\ZZ/\ZZ$ satisfies $(\frac{a}{k},\frac{b}{l},z) \in S$. Thus $\# \{z \in \frac{1}{m}\ZZ/\ZZ \mid (\frac{a}{k},\frac{b}{l},z) \in S\}$ is equal to the number of integers strictly between $0$ and $m - \frac{am}{k} - \frac{bm}{k}$, which is $\lceil m - \frac{am}{k} - \frac{bm}{k} \rceil - 1.$
\end{proof}

Recall that $M = \{r \in (\ZZ/\lcm(k,l,m)\ZZ)^*: r \cdot (H \cap ( S \cup -S)) = H \cap(S \cup -S)\}$.

\begin{lemma}
\label{Proposition r equiv 1 mod lcm(k,m)}
Let $r \in M$ be such that for all elements in $H \cap S$ of the form $(\frac{1}{k}, \frac{1}{l}, z)$ we have $(\frac{r}{k}, \frac{r}{l}, rz) \in H \cap S$. 
Write $r_k$ respectively $r_l$ for the smallest positive representative of $r$ modulo $k$ respectively $l$.
Then $\left\lceil m - \frac{m}{k} - \frac{m}{l} \right\rceil = \left\lceil m - \frac{r_k m}{k} - \frac{r_l m}{l} \right\rceil$. Moreover, if $m \geq \max\{k,l\}$, then $r \equiv 1 \pmod{\lcm(k,l)}$.
\end{lemma}

\begin{proof}
Clearly, $\left\lceil m - \frac{m}{k} - \frac{m}{l} \right\rceil \geq \left\lceil m - \frac{r_k m}{k} - \frac{r_l m}{l} \right\rceil$.
As for all elements in $H \cap S$ of the form $(\frac{1}{k}, \frac{1}{l}, z)$ we have $(\frac{r}{k}, \frac{r}{l}, rz) \in H \cap S$, the map $\left\{z \in \frac{1}{m}\ZZ/\ZZ \mid (\frac{1}{k}, \frac{1}{l}, z) \in S\right\} \to \left\{z \in \frac{1}{m}\ZZ/\ZZ \mid (\frac{r}{k}, \frac{r}{l}, z) \in S\right\}, \, z \mapsto rz$ is well-defined. 
Since multiplication by $r$ is a bijection on $H$, this map is injective, so
$$\#\left\{z \in \frac{1}{m}\ZZ/\ZZ \mid (\frac{1}{k}, \frac{1}{l}, z) \in S\right\} \leq \#\left\{z \in \frac{1}{m}\ZZ/\ZZ \mid (\frac{r}{k}, \frac{r}{l}, z) \in S\right\}.$$
Lemma \ref{lemma: ceiling} then gives 
$\lceil m - \frac{m}{k} - \frac{m}{l}\rceil \leq \lceil m - \frac{r_k m}{k} - \frac{r_l m}{l} \rceil$.

Now suppose $m \geq \max\{k,l\}$.
If $r_k = r_l = 1$ we are done, so suppose $r_k,r_l$ are not both $1$. Without loss of generality we assume $r_k \geq 2$. Let $t = m-\frac{m}{k} - \frac{m}{l}$, $u = m-\frac{r_km}{k} - \frac{r_lm}{l}$.
Then $\lceil t \rceil = \lceil u \rceil$, so
$$1 > t-u = m\left(1-\frac{1}{k} - \frac{1}{l}\right) - m\left(1-\frac{r_k}{k} - \frac{r_l}{l}\right) = m\left(\frac{r_k-1}{k} + \frac{r_l - 1}{l}\right) \geq \frac{m}{k} \geq 1.$$
Contradiction. So $r_k = r_l = 1$.
\end{proof}

In the lemma below, given two sets of integers $A,B$, we write $A \equiv B \pmod{m}$ if for all $a \in A$, there is some $b \in B$ such that $a \equiv b \pmod{m}$, and vice versa.

\begin{lemma}
\label{Lemma: 1<c<m-2}
Suppose $m \geq 3$, let $r \in (\ZZ/m\ZZ)^*$ and let $c$ be an integer such that $1 \leq c \leq m-2$. If $\{1,2,...,c\} \equiv \{r,2r,...,cr\}\pmod{m}$, then $r = 1$.
\end{lemma}

\begin{proof}
Suppose $\{1,...,c\} \equiv \{r,2r,...,cr\}\pmod{m}$. Note that $c<\frac{m}{2}$ or $m-(c+1) < \frac{m}{2}$. Choose $c' \in \{c,m-(c+1)\}$ so that $c' < \frac{m}{2}$. Claim: $\{1,2,...,c'\} \equiv \{r,2r,...,c'r\}\pmod{m}$. Proof claim: we already know this is true for $c' = c$, so it remains to be shown that $\{1,2,...,m-(c+1)\} \equiv \{r,2r,...,(m-(c+1))r\}\pmod{m}$. Since $\{1,...,c\} \equiv \{r,2r,...,cr\}\pmod{m}$, we also have $\{c+1,...,m-1\} \equiv \{(c+1)r,...,(m-1)r\} \pmod{m}$, thus
\begin{align*}
    \{1,2,...,m-(c+1)\} &\equiv \{-(m-1),-(m-2),...,-(c+1)\} \\
    &\equiv \{-(m-1)r,-(m-2)r,...,-(c+1)r\} \\
    &\equiv \{r,2r,...,(m-(c+1))r\}\pmod{m},
\end{align*}
proving the claim.
Suppose $r \neq 1$ and assume without loss of generality that $r$ is an integer and $0 < r < m$. Define $d:= \lfloor \frac{c'}{r}\rfloor$. 
We will get a contradiction with $\{1,2,...,c'\} = \{r,2r,...,c'r\}$ by showing that $d+1 \leq c'$ and $c'<(d+1)r < m$. Since $r \neq 1$ we have $d=\lfloor \frac{c'}{r}\rfloor\leq\frac{c'}{2}<c'$, so $d+1 \leq c'$. 
Since $d > \frac{c'}{r}-1$, we also have $dr>(\frac{c'}{r}-1)r = c'-r$, so $(d+1)r = dr+r >c'$. 
Finally, since $\{1,2,...,c'\} = \{r,2r,...,c'r\}$, we have $r \leq c'$, so $(d+1)r = \lfloor \frac{c'}{r}\rfloor r +r \leq 2c' < m$. 
\end{proof}

\begin{proof}[Proof of Proposition \ref{Proposition: If r in Mklm, then r in pm 1}]
Proposition \ref{Proposition: klmr-Burnside if and only if r in Mklm} gives that $M = M_{\PSL}$, so it suffices to show that if $\frac{1}{k}+\frac{1}{l}+\frac{1}{m}<1$, then $M = \{\pm1\}$.
By Lemma \ref{Lemma: multiplier general} and Proposition \ref{Proposition: klmr-Burnside if and only if r in Mklm} we have that $\{\pm 1 \} \subseteq M$.

Suppose $\frac{1}{k}+\frac{1}{l}+\frac{1}{m}<1$. Clearly $M$ does not change if we permute $k,l,m$, so we may assume without loss of generality that $m \geq \max\{k,l\}$. Then $m \geq 4$. Let $r \in M$.  Claim: at least one of the sets $T:=\{(\frac{r}{k}, \frac{r}{l}, z): z \in \frac{1}{m}\ZZ/\ZZ\} \cap S$ and $T^*:=\{(\frac{r}{k}, \frac{r}{l}, z): z \in \frac{1}{m}\ZZ/\ZZ\} \cap -S$ is empty. 
Proof claim: define $\pi: (\RR/\ZZ)^3 \to (\RR/\ZZ)^2, (x,y,z) \mapsto (x,y)$. Notice that $\pi(T) \cap \pi(T^*) \subseteq \pi(S) \cap \pi(-S) = \emptyset$, so $(\frac{r}{k}, \frac{r}{l}) \notin \pi(T) \cap \pi(T^*)$, proving the claim.

So since $r \in M$, either for all elements in $H \cap S$ of the form $(\frac{1}{k}, \frac{1}{l}, z)$ we have $(\frac{r}{k}, \frac{r}{l}, rz) \in H \cap S$ or for all such elements we have $(\frac{r}{k}, \frac{r}{l}, rz) \in -(H \cap S)$. Choose $r' \in \{r,-r\}$ so that for all elements in $H \cap S$ of the form $(\frac{1}{k}, \frac{1}{l}, z)$ we have $(\frac{r'}{k}, \frac{r'}{l}, r'z) \in H \cap S$. Note that $r' \in M$ by Lemma \ref{Lemma: multiplier general} and Proposition \ref{Proposition: klmr-Burnside if and only if r in Mklm}. Now Lemma \ref{Proposition r equiv 1 mod lcm(k,m)} gives that $r' \equiv 1 \pmod{\lcm(k,l)}$. 

If $(\frac{1}{k}, \frac{1}{l}, z) \in H \cap S$, then $(\frac{1}{k}, \frac{1}{l}, r'z) = (\frac{r'}{k}, \frac{r'}{l}, r'z) \in H \cap S$. This, combined with the fact that multiplication by $r'$ is a bijection on $\frac{1}{m}\ZZ/\ZZ$, gives that multiplication by $r'$ permutes the set of elements in $H \cap S$ of the form $(\frac{1}{k},\frac{1}{l},z)$. 
Thus $\left\{z \in \frac{1}{m}\ZZ/\ZZ \mid (\frac{1}{k}, \frac{1}{l}, (r')^{-1}z) \in S \right\}
= \left\{z \in \frac{1}{m}\ZZ/\ZZ \mid (\frac{1}{k}, \frac{1}{l}, z) \in S \right\}$. Let $c := \lceil m(1-\frac{1}{k} - \frac{1}{l}) \rceil - 1$ and let $q:(\ZZ/\lcm(k,l,m)\ZZ)^* \to (\ZZ/m\ZZ)^*$ be the homomorphism that sends $x$ to $x \pmod{m}$. We want to apply Lemma \ref{Lemma: 1<c<m-2} to $m, \, q(r')$ and $c$. Lemma \ref{lemma: ceiling} gives
\begin{align*}
\left\{\frac{1}{m}, \frac{2}{m},..., \frac{c}{m}\right\}
&= \left\{z \in \frac{1}{m}\ZZ/\ZZ \mid (\frac{1}{k}, \frac{1}{l}, z) \in S \right\} \\
&= \left\{z \in \frac{1}{m}\ZZ/\ZZ \mid (\frac{1}{k}, \frac{1}{l}, (r')^{-1}z) \in S \right\} \\
&= \left\{\frac{r'}{m}, \frac{2r'}{m},..., \frac{c r'}{m}\right\}.
\end{align*}
Since $\frac{1}{k} + \frac{1}{l} + \frac{1}{m} < 1$ we have $\frac{1}{m} < 1 - \frac{1}{k} - \frac{1}{l}$, so $1 < m-\frac{m}{k} - \frac{m}{l}$. Thus $2 \leq \lceil m-\frac{m}{k} - \frac{m}{l} \rceil$ and hence $1 \leq c$. Since $m \geq \max\{k,l\}$, we have $\frac{1}{k} + \frac{1}{l} > \frac{1}{m}$, so $1 - \frac{1}{k} - \frac{1}{l} < 1 - \frac{1}{m} = \frac{m-1}{m}$. Therefore $c < m(1 - \frac{1}{k} - \frac{1}{l}) < m-1$. For all integers $d$ we have $\frac{dr'}{m} = \frac{dq(r')}{m}$, so we may apply Lemma \ref{Lemma: 1<c<m-2} to $m, c$ and $q(r')$. This gives us $q(r') \equiv 1 \pmod{m}$, so $r' \equiv 1 \pmod{m}$, and thus $r \in \{\pm r'\} = \{\pm 1\}$.
\end{proof}

\begin{proof}[Proof of the implication (2) $\Rightarrow$ (3) of Theorem \ref{introduction theorem qB}]
Suppose $\frac{1}{k}+\frac{1}{l}+\frac{1}{m}<1$ and $r \notin \{\pm 1\}$. Then $r \notin M_{\PSL}$ by Proposition \ref{Proposition: If r in Mklm, then r in pm 1}, so $\mathrm{PSL}_2(\RR)$ is not $(k,l,m,r)$-quasi-Burnside by Lemma \ref{Lemma: multiplier general}.
\end{proof}

The proposition below will be used in \S \ref{Section: universal groups}.

\begin{proposition}
\label{Proposition: PSL not hemi-Burnside}
Suppose $k = l$ and let $r \in (\ZZ/\lcm(k,m)\ZZ)^*$. If $\frac{2}{k}+\frac{1}{m} \leq 1$, then there are $x,y,z \in \PSL$ with $xy = z$, $x^k = y^k = z^m = 1$ such that there do not exist $x',x'',z' \in \PSL$ with $x'x''= z'$, $x' \sim x^{-r}$, $x'' \sim x^r$, $z' \sim z^r$.
\end{proposition}

\begin{proof}
Given a group $G$, recall that
$B_G:= \{(C,D,E) \in (G/{\sim})^3: 1 \in CDE, \, C \subseteq G[k]^0, \, D \subseteq G[k]^0, \, E \subseteq G[m]^0\}$.
Analogous to the definition of $M_G$ in \S \ref{Section: Quasi-Burnside and quasi-Honda groups}, we define 
$$N_G:= \{r \in (\ZZ/\lcm(k,m)\ZZ)^*\mid \forall (C,D,E) \in B_G: (C^{-r},C^r,E^r) \in B_G\}.$$ 
Given a group $G$, one can show analogous to Lemma \ref{Lemma: multiplier general}, that $r \in N_G$ if and only if for all $x,y,z \in G$ we have
$$xy = z, \, x^k = y^k = z^m = 1 \Rightarrow \exists x',x'',z' \in G : x'x''= z', \, x' \sim x^{-r}, \, x'' \sim x^r, \, z' \sim z^r.$$
So it suffices to show that if $\frac{2}{k}+\frac{1}{m} \leq 1$, then $r \notin N_{\PSL}$.

Note that since $k = l$, we have $H = (\frac{1}{k}\ZZ/\ZZ) \oplus (\frac{1}{k}\ZZ/\ZZ) \oplus (\frac{1}{m}\ZZ/\ZZ)$. By definition of $N_{\PSL}$ we have that $r \in N_{\PSL}$ if and only if for all $(C_a,C_b,C_c) \in B_{\PSL}$ we have $(C_a^{-r},C_a^r,C_c^r) \in B_{\PSL}$. Since $(C_a^{-r},C_a^r,C_c^r) = (C_{-ra},C_{ra},C_{rc})$, Lemmas \ref{Lemma: bijection H, PSLk/sim} and \ref{Lemma: BPSL2R} give that this is equivalent to the statement that for all $(a,b,c) \in H \cap (S \cup -S \cup T)$ we have $(-ra,ra,rc) \in H \cap(S \cup -S \cup T)$.
Given $s \in \QQ/\ZZ$, we write $[s]$ for the smallest non-negative representative of $s$ modulo $1$. 
Note that given $(a,b,c) \in H \cap(S \cup -S \cup T)$ we have
\begin{align*}
    (-ra,ra,rc) \in H \cap(S \cup -S \cup T) &\Leftrightarrow [-ra] + [ra] + [rc] \leq 1 \text{ or } [-ra] + [ra] + [rc] \geq 2 \\
    &\Leftrightarrow 1-[ra] + [ra] +[rc] \leq 1 \text{ or } 1-[ra] +[ra]+[rc] \geq 2 \\
    &\Leftrightarrow [rc] \leq 0 \text{ or } [rc] \geq 1 \Leftrightarrow rc = 0 \Leftrightarrow c = 0,
\end{align*}
but such $c$ is never equal to 0. Thus $r \in N_{\PSL}$ if and only if $H \cap (S \cup -S \cup T) = \emptyset$.
If $\frac{2}{k}+\frac{1}{m} \leq 1$, then $(\frac{1}{k},\frac{1}{k},\frac{1}{m}) \in H \cap (S \cup -S \cup T)$, so then $r \notin N_{\PSL}$.
\end{proof}

\begin{proposition}
\label{Proposition: PSL2R Honda and quasi-Honda}
The group $\mathrm{PSL}_2(\mathbb{R})$ is Honda and quasi-Honda.
\end{proposition}

\begin{proof}
Theorem 3 of \cite{Thompson} states that if $K$ is a field with more than 3 elements, then every element of $\mathrm{PSL}_n(K)$ is a commutator. Thus $\mathrm{PSL}_n(K)$ is Honda for such $K$, so in particular $\mathrm{PSL}_2(\mathbb{R})$ is Honda.

Let $r$ be an integer, $x,y,z \in \mathrm{PSL}_2(\mathbb{R})$ and suppose that $[x,y] = z$, $\langle x \rangle = \langle x^r \rangle$, $\langle z \rangle = \langle z^r \rangle$. 
We will show that there exists $w \in \mathrm{PSL}_2(\mathbb{R})$ with $[x,w] = z^r$. By Proposition \ref{Proposition: x or z infinite order quasi-Honda} we may assume that $x$ and $z$ have finite order. 
By Lemma \ref{Lemma: PSL TOR}, the non-identity elements in $\mathrm{PSL}_2(\mathbb{R})^{\mathrm{Tor}}$ are exactly the elements conjugate to some $\sigma_a$ where $a \in (\QQ/\ZZ)\setminus \{0\}$.
Suppose $z \neq 1$. Then $x \neq 1$, so $x = {}^g\sigma_a$, $z = {}^h\sigma_c$ with $g,h \in \mathrm{PSL}_2(\mathbb{R})$ and $a,c \in (\QQ/\ZZ)\setminus \{0\}$. We have $1 = x ({}^yx^{-1})z^{-1} = ({}^g\sigma_a) ({}^{yg}\sigma_{1-a}) ({}^h\sigma_{1-c})$, so if we write $[a],[1 - a],[1 -c] \in (0,1)$ for the respective representatives of $a,1-a,1-c$, then Lemma \ref{Lemma: Orevkov} gives $2 - [c] = [a]+[1 - a]+[1 -c] \notin (1,2)$. Contradiction, so $z = 1$. Thus $[x,y] = z = z^r$.
\end{proof} 

\section{Universal groups}
\label{Section: universal groups}

Recall that $B_{k,l,m} = \langle a,c \mid a^k=(a^{-1}c)^l=c^m=1 \rangle$ and $H_{k,m} = \langle a,b,c \mid [a,b] = c,\, a^k = c^m = 1 \rangle$.

We will now determine the structure of $H_{k,m}$ and $B_{k,k,m}$. Define 
$$G_{k,m} := \langle c_i \, (i \in \mathbb{Z}/k\mathbb{Z})\mid c_i^m = 1 \, (i \in \mathbb{Z}/k\mathbb{Z}), \, c_{k-1} \cdots c_1c_0 = 1 \rangle.$$  
One easily checks that there exists a unique $\sigma \in \mathrm{Aut}(G_{k,m})$ such that for all $i \in \mathbb{Z}/k\mathbb{Z}$ one has $\sigma(c_i) = c_{i+1}$. This $\sigma$ defines a semidirect product $G_{k,m} \rtimes \langle a \rangle$ where $a$ has order $k$, and the inner action of $a$ on $G_{k,m}$ is given by $\sigma$. Given this $\sigma$, one readily shows that there exists a unique $\tau \in \mathrm{Aut}(G_{k,m} \conv \langle b \rangle)$, where $b$ has infinite order, with $\tau|_{G_{k,m}} = \sigma$, $ \tau(b) = c_0b$. 
Notice that $\tau^k(b) = c_{k-1} \cdots c_1c_0b = b$, thus $\tau^k = \mathrm{id}$. So $\tau$ defines a semidirect product $(G_{k,m} \conv \langle b \rangle) \rtimes \langle a \rangle$ where $a$ has order $k$, and the inner action of $a$ on $G_{k,m} \conv \langle b \rangle$ is given by $\tau$. 
We deliberately use the same letter $a$ as in $B_{k,k,m}$ and $H_{k,m}$, and the same letter $b$ as in $H_{k,m}$. Proposition \ref{Proposition: structure universal groups} justifies this. In this proposition, $H_{k,m}$ and $B_{k,k,m}$ are written in terms of free and semidirect products, allowing us to easily embed $B_{k,k,m}$ in $H_{k,m}$.

\begin{proposition} \label{Proposition: structure universal groups}
Let $\langle a \rangle$, $\langle b \rangle$ be cyclic groups of orders $k$, $\infty$ respectively. 
\begin{propenum}
\item \label{Proposition: structure universal q-Burnside group}
There exists an isomorphism of groups $G_{k,m} \rtimes \langle a \rangle \xrightarrow{\sim} B_{k,k,m}: a \mapsto a$, $c_i \mapsto a^ica^{-i}$.
\item \label{Proposition: structure universal q-Honda group}
There exists an isomorphism of groups $(G_{k,m} \conv \langle b \rangle) \rtimes \langle a \rangle \xrightarrow{\sim} H_{k,m} : a \mapsto a$, $b \mapsto b$, $c_i \mapsto a^i c a^{-i}$. Moreover, $(G_{k,m} \conv \langle b \rangle) \rtimes \langle a \rangle = \langle a, b, c_i \, (i \in \mathbb{Z}/k\mathbb{Z}) \mid c_i^m = 1 \, (i \in \mathbb{Z}/k\mathbb{Z}), \, a^k=1, \, {}^{a}b = c_0b, \, {}^{a}c_i = c_{i+1} \, (i \in \mathbb{Z}/k\mathbb{Z}) \rangle$.
\item \label{Proposition: embedding}
There exists an embedding $B_{k,k,m} \hookrightarrow H_{k,m}: a \mapsto a$, $c \mapsto c$. 
\end{propenum}
\end{proposition}

\begin{proof}
The proof of (a) is similar to the proof of the first statement of (b), so we omit it. 
It is not hard to see that $(G_{k,m} \conv \langle b \rangle) \rtimes \langle a \rangle \to H_{k,m}: a \mapsto a$, $b \mapsto b$, $c_i \mapsto a^i c a^{-i}$ and $H_{k,m} \to (G_{k,m} \conv \langle b \rangle) \rtimes \langle a \rangle: a \mapsto a$, $b \mapsto b$, $c \mapsto c_0$ extend to inverse homomorphisms, which proves the first statement of (b). 
One easily checks that the relations $a^k = 1$, ${}^ab=c_0b$ and ${}^ac_i = c_{i+1}$ (for all $i \in \mathbb{Z}/k\mathbb{Z}$) imply that $c_{k-1}\cdots c_1c_0=1$, proving the second statement of (b).
Since $H_{k,m} \cong (G_{k,m} \conv \langle b \rangle) \rtimes \langle a \rangle$, $B_{k,k,m} \cong G_{k,m} \rtimes \langle a \rangle$ and the action of $\langle a \rangle$ on $G_{k,m}$ is identical in both groups, we may view $B_{k,k,m}$ as a subgroup of $H_{k,m}$.
\end{proof}

Note that given $s \in \mathbb{Z}/k\mathbb{Z}$, the automorphisms $\sigma^s$ and $\tau^s$ are well-defined, as $\sigma^k = \mathrm{id}$ and $\tau^k = \mathrm{id}$.
If $r \in (\mathbb{Z}/\lcm(k,m)\mathbb{Z})^*$ and $[r]$ is the smallest positive representative of $r$ modulo $\lcm(k,m)$, then we define $c_{r-1}\cdots c_1c_0:=c_{[r]-1}\cdots c_1c_0$.

\begin{lemma}
\label{Lemma: property of cr-1 cdots c1c0}
In $G_{k,m} \rtimes \langle a \rangle$ we have $c_{r-1}\cdots c_1c_0 = a^r \cdot (a^{-1}c_0)^r$.
\end{lemma}

\begin{proof}
For all $s \in \mathbb{Z}/k\mathbb{Z}$ we have $a^s c_0 a^{-s} = \sigma^s(c_0) = c_s$, so 
\[
(a^{-1}c_0)^r = a^{-r}a^{r-1}c_0a^{-(r-1)}a^{r-2}c_0 \cdots a^{-2}a c_0 a^{-1} c_0 = a^{-r}c_{r-1}\cdots c_1c_0. \qedhere
\]
\end{proof}

The following proposition and lemma will be used in \S \ref{Section: Passing to a subgroup} to prove the equivalence (1) $\Leftrightarrow$ (2) of Theorem \ref{introduction theorem qH}.
Recall that if $\gcd(k,m) \le 2$, then given $r \in (\ZZ/\lcm(k,m)\ZZ)^*$ there exists a unique element $r^* \in (\ZZ/\lcm(k,m)\ZZ)^*$ with $r^* \equiv r \pmod{k}$ and $r^* \equiv -r \pmod{m}$ (Lemma \ref{Lemma: Chinese remainder theorem}).

\begin{proposition}
\label{Proposition: universal groups conclusion}
Let $r \in (\mathbb{Z}/\lcm(k,m)\mathbb{Z})^*$.
\begin{propenum}
    \item \label{Proposition: universal groups conclusion qBurnside}
    Every group is $(k,k,m,r)$-quasi-Burnside if and only if there exist $x,y \in G_{k,m}$ and $i \in \mathbb{Z}/k\mathbb{Z}$ such that $\sigma^r(x)c_{r-1}\cdots c_1c_0x^{-1} = {}^yc_i^r$.
    \item \label{Proposition: quasi semi Burnside conclusion}
    Suppose $\gcd(k,m) \leq 2$. Then every group is $(k,k,m,r^*)$-quasi-Burnside if and only if there exist $x,y \in G_{k,m}$ and $i \in \ZZ/k\ZZ$ such that $(c_{r-1} \cdots c_1c_0)^{-1} \sigma^r(x)x^{-1} = {}^yc_i^{r}$.
    \item \label{Proposition: universal groups conclusion qHonda}
    Every group is $(k,m,r)$-quasi-Honda if and only if there exist $v,w \in G_{k,m} \conv \langle b \rangle$ and $i \in \mathbb{Z}/k\mathbb{Z}$ such that $\tau^r(v)v^{-1} = {}^wc^r_i$.
\end{propenum}
\end{proposition}

\begin{proof}
The proofs of (a) and (c) are similar, so we omit the proof of (c).
Suppose every group is $(k,k,m,r)$-quasi-Burnside. Then by Proposition \ref{Proposition: universal q-Burnside group}, there exist $g,h \in B_{k,k,m}$ such that $a^r \cdot {}^g(a^{-1}c)^r = {}^hc^r$, so by \ref{Proposition: structure universal q-Burnside group} there exist $g,h \in G_{k,m} \rtimes \langle a \rangle$ such that $a^r \cdot {}^g(a^{-1}c_0)^r = {}^hc_0^r$. 
Choose such $g,h$, and write $g = a^j\cdot x$ with $x \in G_{k,m}$ and $j \in \ZZ/k\ZZ$. Then $a^r \cdot {}^{a^j\cdot x}(a^{-1}c_0)^r = {}^hc_0^r$. 
By conjugating both sides by $a^{-j}$, and writing $a^{-j}\cdot h = y\cdot a^i$ with $y \in G_{k,m}$ and $i \in \ZZ/k\ZZ$, we find $a^r \cdot {}^x(a^{-1}c_0)^r = {}^{a^{-j}h}c_0^r = {}^{ya^i}c_0^r$.
By Lemma \ref{Lemma: property of cr-1 cdots c1c0} we have $a^r \cdot {}^x (a^{-r} c_{r-1} \cdots c_1c_0) = {}^{ya^i}c_0^r$, so $\sigma^r(x)c_{r-1}\cdots c_1c_0x^{-1} = {}^{ya^i}c_0^r = {}^yc_i^r$.
The entire argument is reversible, proving statement (a).

Suppose $\gcd(k,m) \leq 2$. By (a), every group being $(k,k,m,r^*)$-quasi-Burnside is equivalent to the existence of $x,y \in G_{k,m}$ and $i \in \ZZ/k\ZZ$ such that $\sigma^r(x)c_{r-1}\cdots c_1c_0x^{-1} = {}^yc_i^{-r}$
Choose such $x,y,i$.
Inverting both sides of the equation, and then conjugating by $x^{-1}$ gives $(c_{r-1} \cdots c_1c_0)^{-1} \sigma^r(x^{-1})x = {}^{x^{-1}y}c_i^{r}$, so for $x' = x^{-1}$, $y' = x^{-1}y$ we have 
$(c_{r-1} \cdots c_1c_0)^{-1} \sigma^r(x')(x')^{-1} = {}^{y'}c_{i}^{r}.$
This proves one implication of (b). By reading the argument backwards, one proves the other implication.
\end{proof}

\begin{lemma}
\label{Lemma: conclusion hemi demi}
Let $r \in (\ZZ/\lcm(k,m)\ZZ)^*$ and suppose $\frac{2}{k}+\frac{1}{m} \leq 1$. Then there do not exist $x,y \in G_{k,m}$ and $i \in \ZZ/k\ZZ$ such that $\sigma^r(x)x^{-1} = {}^yc_i^{r}$ or $(c_{r-1} \cdots c_1c_0)^{-1} \sigma^r(x)c_{r-1} \cdots c_1c_0x^{-1} = {}^yc_i^r$.
\end{lemma}

\begin{proof}
Analogously to Proposition \ref{Proposition: universal groups}, one can show that every group $G$ has the property that for all $x,y,z \in G$ we have
$$xy = z, \, x^k = y^k = z^m = 1 \Rightarrow \exists x',x'',z' \in G : x'x''= z', \, x' \sim x^{-r}, \, x'' \sim x^r, \, z' \sim z^r$$ 
if and only if there exist $g,h \in B_{k,k,m}$ such that $a^{-r} \cdot {}^ga^r = {}^hc^r$. By Proposition \ref{Proposition: PSL not hemi-Burnside}, the group $\PSL$ does not have the property above, so such $g,h$ do not exist.
If $x,y \in G_{k,m}$ and $i \in \ZZ/k\ZZ$ satisfy $\sigma^r(x)x^{-1} = {}^yc_i^r$, then $a^r x a^{-r} x^{-1} = {}^{ya^i}c_0^r$. Conjugating both sides by $x^{-1}a^{-r}$ then gives $a^{-r}\cdot{}^{x^{-1}}a^r = {}^{x^{-1}a^{-r}ya^i}c_0^r$. Then by Proposition \ref{Proposition: structure universal q-Burnside group} there exist $g,h$ as above, giving a contradiction and proving the first statement of the lemma.

Let $x,y \in G_{k,m}$, $i \in \ZZ/k\ZZ$ such that $(c_{r-1} \cdots c_1c_0)^{-1} \sigma^r(x)c_{r-1} \cdots c_1c_0x^{-1} = {}^yc_i^r$. By Lemma \ref{Lemma: property of cr-1 cdots c1c0} we have ${}^yc_i^r = (c_{r-1} \cdots c_1c_0)^{-1} \sigma^r(x)c_{r-1} \cdots c_1c_0x^{-1} = ((a^{-1}c_0)^{-r}a^{-r})a^rxa^{-r}(a^r(a^{-1}c_0)^r)x^{-1}.$
Thus ${}^{ya^i}c_0^r = {}^yc_i^r = (a^{-1}c_0)^{-r}x(a^{-1}c_0)^rx^{-1}$.
By Proposition \ref{Proposition: structure universal q-Burnside group} there now exist $v,w \in B_{k,k,m}$ such that $(a^{-1}c)^{-r} \cdot {}^v(a^{-1}c)^r = {}^wc^r$. Define $\alpha: B_{k,k,m} \to B_{k,k,m}$, $a \mapsto a^{-1}c$, $c \mapsto a^{-1}ca$. One easily shows that $\alpha$ is a homomorphism and that $\alpha^2$ is the identity, so $\alpha$ is an automorphism of order 2 that sends $a^{-1}c$ to $a$. 
Applying $\alpha$ to our formula gives $a^{-r} \cdot {}^{\alpha(v)}a^r = {}^{\alpha(w)a^{-1}}c^r$, but in the beginning of the proof we showed that such elements $g:=\alpha(v)$, $h:=\alpha(w)a^{-1}$ do not exist.
\end{proof}

\section{Reduced words}
\label{Section: Reduced words}

In this section we define a special notion of a reduced word in a free product of a group and an infinite cyclic group. This special notion has the advantage that under certain automorphisms of this free product, the ``lengths'' of these reduced words are preserved. (Proposition \ref{Proposition: reduced word of tau(v)}.) We will use this special notion to prove a combinatorial group theoretic result (Proposition \ref{Proposition: elliminating b v2}), with which we can prove the equivalence (1) $\Leftrightarrow$ (2) of Theorem \ref{introduction theorem qH}.

Throughout this section, $G$ will be a group and $\langle b \rangle$ will be an infinite cyclic group.

\begin{definition}
Let $s$ be a non-negative integer. A reduced word of length $s$ in $G \conv \langle b \rangle$ is a sequence $(u_0, u_1,...,u_{2s})$ with $u_{2i} \in G$ for $0 \leq i \leq s$ and $u_{2i+1} \in \{b, b^{-1}\}$ for $0 \leq i < s$ such that there does not exist $i$ with $0<i<s$, $u_{2i} = 1$ and $u_{2i-1} = u_{2i+1}^{-1}$.
\end{definition}

Note that the length of a reduced word is the number of its letters that are equal to $b$ or $b^{-1}$.

\begin{proposition}
\label{Proposition: van der Waerden}
For each non-negative integer $s$, let $Y_s$ be the set of reduced words of length $s$ in $G \conv \langle b \rangle$. Then for every $g \in G \conv \langle b \rangle$ there is a unique non-negative integer $s$ and a unique $(u_0,...,u_{2s}) \in Y_s$ such that $g = u_0 u_1 \cdots u_{2s}$. \hfill \qedsymbol{}
\end{proposition}

We say that $(u_0,...,u_{2s})$ is the reduced word of $g$.

\begin{proof}
We follow a method that J.-P. Serre used in the proof of Theorem 1 in section 1.2 of \cite{Serre}, originally invented by B.L. van der Waerden \cite{vanderWaerden}.

Let $C = G \conv \langle b \rangle$ and let $Y = \bigsqcup_{s \in \mathbb{Z}_{\geq 0}} Y_s$ be the disjoint union of all $Y_s$'s. We will let $C$ act on $Y$ by letting $G$ and $\langle b \rangle$ act on $Y$ separately. For every non-negative integer $s$, we define the action of $G$ on $Y$ by the maps
$$\alpha_s: G \times Y_s \to Y_s: (v,(u_0,...,u_{2s})) \mapsto (vu_0,...,u_{2s}).$$
For all $\epsilon \in \{\pm 1\}$ we define the map 
$$\beta_\epsilon: Y \to Y: (u_0,...,u_{2s}) \mapsto \begin{cases} (u_2,...,u_{2s}), & \text{if } s \geq 1, \, u_0 = 1, \, u_1 = b^{-\epsilon}, \\
    (1,b^{\epsilon},u_0,...,u_{2s}), & \text{else}.
\end{cases}$$
One easily checks that $\beta_{-1}\beta_1 = \beta_1\beta_{-1} = \mathrm{id}$, thus $\beta_1$ is a permutation of $Y$. Hence we may define the action of $\langle b \rangle$ by the group homomorphism $\langle b \rangle \to \mathrm{Sym}(Y)$ sending $b$ to $\beta_1$.

We finish the proof by showing that the map $\theta: Y \to C$ sending  $(u_0,...,u_{2s})$ to $u_0 \cdots u_{2s}$ is a bijection. 
For all $w \in Y$ and all $h \in G \cup \{b,b^{-1}\}$ we have $\theta(hw) = h\theta(w)$, so $\theta(gw) = g \theta(w)$ for all $g \in C$. Thus for all $g \in C$ we have $g = g\cdot \theta((1)) = \theta(g\cdot (1))$, where $(1) \in Y$ is the word consisting of only $1\in G$, so $\theta$ is surjective. One can show, by induction on the length of $w$, that for all $w \in Y$ one has $\theta(w)\cdot(1) = w$. So if $\theta(w) = \theta(w')$, then $w = \theta(w) \cdot (1) = \theta(w') \cdot (1) = w'$, so $\theta$ is injective. 
\end{proof}

\begin{lemma}
\label{Lemma: middle letter}
Let $(u_i)_{i=0}^{2s}$, $(v_i)_{i=0}^{2s'}$ be reduced words in $G \conv \langle b \rangle$ with $s,s'$ non-negative integers. If $(u_0,...,u_{2s-1},u_{2s}v_0,v_1,...,v_{2s'})$ is not a reduced word in $G \conv \langle b \rangle$, then $s\geq 1$, $s' \geq 1$, $u_{2s-1} = v_1^{-1}$, and $u_{2s} v_0 = 1$. \hfill \qedsymbol{}
\end{lemma}

\begin{proof}
Suppose that $w:=(u_0,...,u_{2s-1},u_{2s}v_0,v_1,...,v_{2s'})$ is not a reduced word. Then there exists a triplet $(w_1,w_2,w_3)$ of consecutive elements of $w$ for which one has $w_2 = 1$ and $w_1 = w_3^{-1} \in \{b,b^{-1}\}$. Since $(u_i)_{i=0}^{2s}$, $(v_i)_{i=0}^{2s'}$ are reduced words, $w_1,w_2,w_3$ are not all three elements of $(u_i)_{i=0}^{2s}$, and not all three elements of $(v_i)_{i=0}^{2s'}$. So $s \geq 1$, $s' \geq 1$, $w_1 = u_{2s-1}$, $w_2 = u_{2s}v_0$ and $w_3 = v_1$.
\end{proof}

\begin{definition}
We define $\mathrm{len}: G \conv \langle b \rangle \to \mathbb{Z}_{\geq 0}$ to be the function that sends an element of $G \conv \langle b \rangle$ to the length of its unique reduced word given by Proposition \ref{Proposition: van der Waerden}.
\end{definition}

\begin{proposition}
\label{Proposition: inverse is lengthpreserving}
For all $u \in G \conv \langle b \rangle$ we have $\mathrm{len}(u) = \mathrm{len}(u^{-1})$. More precisely, if $(u_i)_{i=0}^{2s}$ is the reduced word of some $u \in G \conv \langle b \rangle$ with $s$ a non-negative integer, then $(u_i^{-1})_{2s}^0$ is the reduced word of $u^{-1}$. \hfill \qedsymbol{}
\end{proposition}

\begin{lemma}
\label{Lemma: middle letters}
Let $u,v \in G \conv \langle b \rangle$ and let $(u_i)_{i=0}^{2s}$, $(v_i)_{i=0}^{2s'}$ be their respective reduced words with $s,s'$ non-negative integers. Then there exists an integer $0 \leq n \leq \min\{s,s'\}$ such that the sequence $w_n:=(u_{0},..., u_{2s-2n-1},u_{2s-2n}v_{2s'-2n}^{-1}, v_{2s'-2n-1}^{-1},...,v_0^{-1})$ is a reduced word. Moreover, for the smallest such integer $n$, we have that $w_n$ is the reduced word of $uv^{-1}$, and $\mathrm{len}(uv^{-1}) = s+s'-2n$.
\end{lemma}

Note that $2n$ is the number of pairs of letters that cancel each other in $uv^{-1}$.

\begin{proof}
Let $0 \leq n \leq \min\{s,s'\}$ be an integer.
Proposition \ref{Proposition: inverse is lengthpreserving} gives that $(v_i^{-1})_{2s'}^0$ is a reduced word, so if $n = \min\{s,s'\}$, then Lemma \ref{Lemma: middle letter} gives that $w_n$ is a reduced word. Suppose $n$ is the smallest integer such that $w_n$ is a reduced word. Then for all $i$ with $0 \leq i < n$ we have that
$(u_{0},..., u_{2s-2i-1},u_{2s-2i}v_{2s'-2i}^{-1}, v_{2s'-2i-1}^{-1},...,v_0^{-1})$
is not a reduced word, so then Lemma \ref{Lemma: middle letter} gives $u_{2s-2i-1}u_{2s-2i}v_{2s'-2i}^{-1} v_{2s'-2i-1}^{-1} = 1$. Thus 
$uv^{-1} = u_{0}\cdots u_{2s-2n-1}\cdot u_{2s-2n} v_{2s'-2n}^{-1} \cdot v_{2s'-2n-1}^{-1}\cdots v_0^{-1},$
so $w_n$ is the reduced word of $uv^{-1}$. Now $\mathrm{len}(uv^{-1}) = s+s'-2n$.
\end{proof}

\section{Passing to a subgroup}
\label{Section: Passing to a subgroup}

In this section, we use the special notion of a reduced word from the last section to prove that there exists a solution of a certain equation in the free product of some group $G$ with an infinite cyclic group $\langle b \rangle$ if and only if there exists a solution of at least one of four equations in $G$ (Proposition \ref{Proposition: elliminating b v2}). 
These equations are very similar to those in the theory about when certain equations in $G_{k,m}$ and $G_{k,m}\conv \langle b \rangle$ have solutions (Proposition \ref{Proposition: universal groups conclusion} and Lemma \ref{Lemma: conclusion hemi demi}). 
This allows us to finish the proof of Theorem \ref{introduction theorem qH}, by applying Proposition \ref{Proposition: elliminating b v2} to our group $G_{k,m}$, and then using the theory from \S \ref{Section: universal groups}.

Analogously to the existence of $\tau$ just before Proposition \ref{Proposition: structure universal groups} we have that given a group $G$, an infinite cyclic group $\langle b \rangle$, $\varphi \in \mathrm{Aut}(G)$, and $p \in G$, there exists a unique $\psi \in \mathrm{Aut}(G \conv \langle b \rangle)$ with $\psi|_G = \varphi$, $\psi(b) = pb$.

\begin{proposition}
\label{Proposition: reduced word of tau(v)}
Let $\varphi \in \mathrm{Aut}(G)$, $p \in G$, and let $\psi$ be the unique automorphism of $G \conv \langle b \rangle$ with $\psi|_G = \varphi$, $\psi(b) = pb$. Then for all $v \in G \conv \langle b \rangle$ we have $\mathrm{len}(v) = \mathrm{len}(\psi(v))$. More precisely, if $(v_i)_{i=0}^{2s}$ is the reduced word of $v$ with $s$ a non-negative integer, define for all $0 \leq i < s$ the number $\epsilon_{2i+1} \in \{\pm1\}$ to be such that $v_{2i+1} = b^{\epsilon_{2i+1}}$, define
$$p_\epsilon = \begin{cases}
    p & \text{ if } \epsilon = 1, \\
    1 & \text{ if } \epsilon = -1,
\end{cases}$$
and define
$$\begin{cases}
u_0 = \varphi(v_0) p_{\epsilon_1}, \\
u_{2i} = (p_{-\epsilon_{2i-1}})^{-1} \varphi(v_{2i}) p_{\epsilon_{2i+1}} & \text{ if } 0<i<s, \\
u_{2i+1} = v_{2i+1} & \text{ if } 0 \leq i < s, \\
u_{2s} = (p_{-\epsilon_{2s-1}})^{-1} \varphi(v_{2s}).
\end{cases}$$
Then $(u_i)_{i=0}^{2s}$ is the reduced word of $\psi(v)$.
\end{proposition}

\begin{proof}
One straightforwardly checks that $\psi(v) = u_0 \cdots u_{2s}$, and that $(u_i)_{i=0}^{2s}$ is a reduced word.
\end{proof}

\begin{proposition}
\label{Proposition: elliminating b v2}
Let $G$ be a group, $\langle b \rangle$ an infinite cyclic group, $\varphi \in \mathrm{Aut}(G)$, $p,q \in G$, and let $\psi$ be the unique automorphism of $G \conv \langle b \rangle$ with $\psi|_G = \varphi$, $\psi(b) = pb$. Then there exist $v,w \in G \conv \langle b \rangle$ such that $\psi(v)v^{-1} = {}^wq$ if and only if there exist $x,y \in G$ such that 
$$\varphi(x)x^{-1}={}^yq \text{, or } \varphi(x)px^{-1}={}^yq \text{, or } p^{-1}\varphi(x)x^{-1}={}^yq \text{, or } p^{-1}\varphi(x)px^{-1}={}^yq.$$
\end{proposition}

\begin{proof}
To prove the implication $\Leftarrow$, choose $v = x$, $w = y$ in the first case, $v=xb$, $w = y$ in the second case, $v = b^{-1}x$, $w = b^{-1}y$ in the third case and $v = b^{-1}xb$, $w = b^{-1}y$ in the fourth case.

As for the other implication: if $q = 1$, then any $y \in G$ satisfies $\varphi(1)1^{-1} = {}^yq$, so we may assume that $q \neq 1$. Let $v,w \in G \conv \langle b \rangle$ be such that $\psi(v)v^{-1} = wqw^{-1}$. Note that $\mathrm{len}(\psi(v)) = \mathrm{len}(v)$ by Proposition \ref{Proposition: reduced word of tau(v)}. Let $(u_i)_{i=0}^{2s}$, $(v_i)_{i=0}^{2s}$, $(w_i)_{i=0}^{2t}$ be the reduced words of respectively $\psi(v)$, $v$, $w$. As $q \neq 1$, we have $w_{2t}qw_{2t}^{-1} \neq 1$. 
Clearly $(w_0,...,w_{2t-1},w_{2t}q)$ is a reduced word, and by Proposition \ref{Proposition: inverse is lengthpreserving} the sequence $(w_i^{-1})_{i=2t}^0$ is a reduced word. Lemma \ref{Lemma: middle letter} then gives that $(w_0,...,w_{2t}qw_{2t}^{-1},...,w_0^{-1})$ is the reduced word of $wqw^{-1}$. 

We want to apply Lemma \ref{Lemma: middle letters} to $\psi(v)$ and $v$ to find the reduced word of $\psi(v)v^{-1}$. Since $\psi(v)v^{-1}  = wqw^{-1}$, we have $\mathrm{len}(\psi(v)v^{-1}) = \mathrm{len}(wqw^{-1}) = 2t$. Thus $n$ as in Lemma \ref{Lemma: middle letters} is equal to $s-t$. We have $u_{2s-2(s-t)} = u_{2t}$ and $v_{2s - 2(s-t)} = v_{2t}$, so Lemma \ref{Lemma: middle letters} gives that $(u_0,...,u_{2t-1},u_{2t}v_{2t}^{-1}, v_{2t-1}^{-1},...,v_0^{-1})$ is the reduced word of $\psi(v)v^{-1}$. As $\psi(v)v^{-1}  = wqw^{-1}$, the middle letters of their reduced words are equal, so $u_{2t}v_{2t}^{-1} = w_{2t}qw_{2t}^{-1}$.
Proposition \ref{Proposition: reduced word of tau(v)} gives 
$u_{2t} \in \{\varphi(v_{2t}), \varphi(v_{2t})p, p^{-1}\varphi(v_{2t}), p^{-1}\varphi(v_{2t})p\}$, so 
\[
w_{2t}qw_{2t}^{-1} = u_{2t}v_{2t}^{-1} \in \{\varphi(v_{2t})v_{2t}^{-1}, \varphi(v_{2t})pv_{2t}^{-1}, p^{-1}\varphi(v_{2t})v_{2t}^{-1}, p^{-1}\varphi(v_{2t})pv_{2t}^{-1}\}. \qedhere
\]
\end{proof}

\begin{lemma}
\label{Lemma: passing to subgrp our case}
Let $r \in (\ZZ/\lcm(k,m)\ZZ)^*$, let $i \in \mathbb{Z}/k\mathbb{Z}$, let $\langle b \rangle$ be an infinite cyclic group, and let $G_{k,m}$, $\sigma$, $\tau$, $c_{r-1} \cdots c_1c_0$ be defined as in \S \ref{Section: universal groups}. Suppose $\frac{2}{k}+\frac{1}{m} \le 1$.
\begin{lemenum}
\item \label{Lemma: passing to subgrp general}
Then there exist $v,w \in G_{k,m} \conv \langle b \rangle$ such that $\tau^r(v)v^{-1} = {}^wc_i^r$ if and only if there exist $x,y \in G_{k,m}$ such that $\sigma^r(x)c_{r-1} \cdots c_1c_0x^{-1}={}^yc_i^r$ or $(c_{r-1} \cdots c_1c_0)^{-1}\sigma^r(x)x^{-1}={}^yc_i^r$.
\item \label{Lemma: passing to subgrp gcd geq 3}
If $\gcd(k,m) > 2$, then there exist $v,w \in G_{k,m} \conv \langle b \rangle$ such that $\tau^r(v)v^{-1} = {}^wc_i^r$ if and only if there exist $x,y \in G_{k,m}$ such that $\sigma^r(x)c_{r-1} \cdots c_1c_0x^{-1}={}^yc_i^r$.
\end{lemenum}
\end{lemma}

\begin{proof}
Applying Proposition \ref{Proposition: elliminating b v2} to $G = G_{k,m}$, $\varphi = \sigma^r$, $p = c_{r-1} \cdots c_1c_0$, $q = c_i^r$, $\psi = \tau^r$ gives that there exist $v,w \in G_{k,m} \conv \langle b \rangle$ such that $\tau^r(v)v^{-1} = {}^wc_i^r$ if and only if there exist $x,y \in G_{k,m}$ such that 
$\sigma^r(x)x^{-1}={}^yc_i^r \text{, or } \sigma^r(x)c_{r-1} \cdots c_1c_0x^{-1}={}^yc_i^r \text{, or } (c_{r-1} \cdots c_1c_0)^{-1}\sigma^r(x)x^{-1}={}^yc_i^r \text{, or } (c_{r-1} \cdots c_1c_0)^{-1}\sigma^r(x)c_{r-1} \cdots c_1c_0x^{-1}={}^yc_i^r$. By Lemma \ref{Lemma: conclusion hemi demi}, there do not exist $x,y \in G_{k,m}$ such that $\sigma^r(x)x^{-1}={}^yc_i^r$ or $(c_{r-1} \cdots c_1c_0)^{-1}\sigma^r(x)c_{r-1} \cdots c_1c_0x^{-1}={}^yc_i^r$, proving (a).

Let $\langle \alpha \rangle$ be a cyclic group of order $\gcd(k,m)$, and define the group homomorphism $\chi: G_{k,m} \to \langle \alpha \rangle$ so that for all $i \in \mathbb{Z}/k\mathbb{Z}$ one has $\chi(c_i) = \alpha$. Notice that $\chi \circ \sigma^r = \chi$. 
Suppose $\gcd(k,m) > 2$. If there are $x,y \in G_{k,m}$ such that $(c_{r-1} \cdots c_1c_0)^{-1}\sigma^r(x)x^{-1}={}^yc_i^r$, then applying $\chi$ would give $\alpha^{-r} = \alpha^r$, so then $\alpha^{2r} = 1$. This contradicts $\gcd(k,m) > 2$, proving (b).
\end{proof}

\begin{proof}[Proof of the equivalence (1) $\Leftrightarrow$ (2) of Theorem \ref{introduction theorem qH}]
If $\frac{2}{k}+\frac{1}{m} \ge 1$, then by Theorem \ref{introduction theorem qB} every group is $(k,k,m,r)$-quasi-Burnside, so by Proposition \ref{Proposition: klmr quasi-Burnside implies kmr quasi-Burnside} every group is also $(k,m,r)$-quasi-Honda.

Suppose that $\frac{2}{k}+\frac{1}{m} < 1$.
We prove the equivalence in the case where $\gcd(k,m) >2$ and in the case where $\gcd(k,m) \le 2$.
By Proposition \ref{Proposition: universal groups conclusion qHonda}, every group is $(k,m,r)$-quasi-Honda if and only if there exist $v,w \in G_{k,m}$ and $i \in \ZZ/k\ZZ$ such that $\tau^r(v)v^{-1} = {}^wc_i^r$. 

If $\gcd(k,m) >2$, then Lemma \ref{Lemma: passing to subgrp gcd geq 3} gives that the existence of such $v,w,i$ is equivalent to the existence of $x,y \in G_{k,m}$ and $i \in \ZZ/k\ZZ$ such that $\sigma^r(x)c_{r-1} \cdots c_1c_0x^{-1}={}^yc_i^r$. By Proposition \ref{Proposition: universal groups conclusion qBurnside} this is equivalent to every group being $(k,k,m,r)$-quasi-Burnside.
If $\gcd(k,m) \leq 2$, then Lemma \ref{Lemma: passing to subgrp general} gives the existence of $v,w,i$ as above is equivalent to the existence of $x,y \in G_{k,m}$ and $i \in \ZZ/k\ZZ$ such that 
$\sigma^r(x)c_{r-1} \cdots c_1c_0x^{-1}={}^yc_i^r$ or $(c_{r-1} \cdots c_1c_0)^{-1}\sigma^r(x)x^{-1}={}^yc_i^r.$
Propositions \ref{Proposition: universal groups conclusion qBurnside} and \ref{Proposition: quasi semi Burnside conclusion} give that this is equivalent to every group being $(k,k,m,r)$-quasi-Burnside or every group being $(k,k,m,r^*)$-quasi-Burnside.
\end{proof}

\textbf{Acknowledgements.} 
I am extremely grateful to Hendrik Lenstra for his invaluable guidance during the writing of this paper, as well as for his excellent supervision over my bachelors and masters theses, which were on the same subject.
I thank Kevin van Yperen for providing the reference to the paper by S.Yu.\ Orevkov.
Thanks also to Ben Martin for pointing out that S.J. Pride gave a family of groups that are not Honda.

%\printbibliography

\bibliographystyle{plainurl}
\bibliography{References}

\end{document}